\documentclass[12pt,a4paper]{amsart}
\usepackage{amssymb}
\usepackage{amsmath,amsthm}
\usepackage{a4wide}

\newcommand{\NN}{\mathbb N}

\newcommand{\CC}{\mathbb C}
\newcommand{\RR}{\mathbb R}

\newtheorem{theorem}{Theorem}[section]
\newtheorem{proposition}[theorem]{Proposition}

\newtheorem{lemma}[theorem]{Lemma}
\newtheorem{definition}[theorem]{Definition}

\newtheorem{remark}[theorem]{Remark}

\begin{document}

%\begin{document}

%%%%%%%%%%%%%%%%%%%%%%%%%%%%%%%%%%%%%%     

%%%%%%%%%%%%%%%%%%%%%%%%%%%%%%%%%%%%%
\title{Bj\"{o}rk- Debrowere algebras, topologies and comparison of regularities}

%%%%%%%%%%%%
\author{A. Khelif}
\author{D. Scarpalezos}
\address{Universite Paris 7\\
	Paris\\ France}
\maketitle
%%%%%%%%%%%%%%%%%%%%%%%%%

%%%%%%%%%%%%%%%%

%%%%%%%%%%%%%%%

%%%%%%%%

\section{Introduction }
%%%%%%%%%
In \cite{Bj.1} was presented a theory of ultradifferentiable functions, using "weight" functions instead of $(M_p)$ sequences as in most papers of Roumieu, Beurling or in Komatsu papers \cite{Komatsu.1}, \cite{Komatsu.2}, \cite{Komatsu.3}.
 This approach of "weight functions"  has also been followed by many other mathematicians such as D. Voigt, R. Meire, M. Langenbrush, A. Taylor and others, who obtained a great number of new interesting results. An open problem was the embedding of ultradistributions in some algebra in a simililar way as the embedding of distributions in Colombeau algebra \cite{Col.1}, \cite{Ober.1}. Such embeddings are usefull if we have to deal both with linear and non linear equations or with linear equations with irregular coefficients. 
 The spaces of duals of all those spaces of ultradifferentiable functions called ulradistributions were embedded in Colombeau type algebras in the doctoral thesis of A. Debrouwere and those results were published also in \cite{D.V.V.} where were exposed and investigated various properties of the algebras introduced in the thesis of A.Debrouwere.
   
     In this paper we investigate properties of the family of weight functions used in some of those  constructions.and especially in \cite{D.V.V.}.especially in the "weight function" models
      We establish in the algebra introduced in \cite{D.V.V.}, topological sharp structures analogous to the ones introduced in Colombeau algebra (algebra where are embedded usual distributions) (see \cite{S.1}, \cite{S.2}, \cite{S.3},\cite{S.4} \cite{G.1}, \cite{G.2} as well as in other analogous contructions \cite{D.S.1} and \cite{D.S.2} and \cite{P.R.S.Z}.
    We also  give projective representations of the Roumieu case and define strong and strict associations of elements of our algebras with ultradistributions and finally we establish results of comparison of regularities analogous to what was done in \cite{S.1}, \cite{P.S.V.} ,\cite{P.R.S.Z}
    In the last section we present results concerning criteria for equalities analogous to what is done in  \cite{P.S.V.} and
    more recently in \cite{P.R.S.Z}.
    
     In order to make the paper almost self contained we presented our own proofs of properties we need also concerning the ultradifferentiable functions.(results also existing in other papers as \cite{D.P.V.}). 
      For many assertions, there are clearly also other proofs, but we present ours in order to favour  a clearer understanding of our methods.For the properties of ultradifferentiable fuctions and their duals , one could louk at the following \cite{B.M.T.}, \cite{R.S.}, \cite{N.V.}.
     
      Of course the litterature on the subject is much richer 
      but those papers can help to get a first  understanding of the subject. 
      Analogous results to ours are proved in the sequence model 
      in \cite{P.R.S.Z}
     but often  with different techniques .                             =
\section{weight functions}

We will consider the family $\mathcal{W}$ of weight functions used in \cite{D.V.V.} for Roumieu case also
with the stronger property
 $ lim \frac{\omega(t)}{log(1+t)}\,=\,\infty .$

\begin{definition}
Let $\,\mathcal{W}\,$ be the spaces of all nondecreasing positive  continuous functions $\,\omega\,$ defined on
 $\RR^+$,such that 
 \\
 a)$\omega(0)\,=\,0 $,and has the property of subadditivity , i.e.: for all couple $(x,y)$ of positive numbers 
 $$\omega(x+y)\,\leq\,\omega(x)\,+\,\omega(y)\,$$
 \\

b) $$\,\int_1^\infty \frac{\omega(t)}{t^2}dt\,<\,\infty . $$

c) $$\lim(\frac{\omega(t)}{log(1+t)})\,=\,\infty.$$

In this family of functions, we establish the following weak equivalence relation:
 two weight functions$\,\omega_1\,$ and
$ \,\omega_2 \,$ are said to be weakly equivalent if there exist positive constants $k$ and $m$ such that 
$$\forall t \in \RR^+\,, k\omega_1(t) \,\leq\,\omega_2(t)\,
\leq\,m\omega_1(t)\,.$$
\\
 We will also define a notion of strong inequality in
  $\,\mathcal{W}\,.$
  \\
  $\omega_1\,<<\,\omega_2\,$ if and only if 
  $\,\omega_1(t) \,=\,o(\omega_2(t)\,.$
  
  \end{definition}
 
  We first establish some results concerning weak equivalence:

\begin{theorem}
a) For any weight function $\,\omega\,\in\,\mathcal{W}\,$
there exists a function $f$ in the same weak equivalence class as $\omega$ which is increasing and such that: 
$\,\frac{f(x)}{x}\,=\,\phi(x)$ is decreasing.
More over any positive continuous function $f$ on  
$\RR^+$ such that$\,f(0)\,=\,0$ ,  equivalent to an element of $\mathcal{W}\,$ and satisfying the property that it is increasing and $\,\frac{f(x)}{x}\,=\,\phi(x)\,$ is decreasing belongs to $\mathcal{W}$.

b)In every weak equivalence class there exists a concave element.
\end{theorem}
\begin{proof}

a)Consider $\,f(x)\,=\int_1^2\frac{\omega(xy}{y^2}dy \,$

By change of variable we obtain: 
$\,\frac{f(x)}{x}\,
=\,\int_x^{2x}\frac{\omega(t)}{t^2}dt\,$

Let $\,z\,\in\,(x,2x)\,$
Then 
$$ \frac{f(x)}{x}\,-\,\frac{f(z)}{z}
\,=\,\int_x^z\frac{\omega(t)}{t^2}dt\,
-\,\int_{2x}^{2z}\frac{\omega(t)}{t^2}\,
  =\,
  \int_x^z(\frac{\omega(s)}{s^2}-\frac{\omega(2s}{2s^2})ds\,
 =\,\int_x^z\frac{2\omega(s)-\omega(2s)}{2s^2}\,
  \geq\,0\,$$ 
  because $\,\omega\,$ is subadditive!
  Notice now that, as $\omega$ is subadditive,
  $$\,\frac{1}{2}\omega(x)\,
  =\,\int_1^2\frac{\omega(x)}{y^2}dy\,\leq\,\int_1^2\frac{\omega(xy)}{y^2}dy\,\leq\,,\int_1^2\frac{\omega(2x)}{y^2}dy\,=
  \,\frac{\omega(2x)}{2}\,\leq\omega(x)\,.                   $$

  Thus $\,f\,$ and $\,\omega\,$ are weakly equivalent.
Let us now remark that one can easily prove that if $f$ is a function such that $ \phi(x)\,= \,\frac{f(x)}{x}$ is decreasing then $f$ is subadditive. The other conditions in order to belong to $\mathcal{W}$ follow from weak equivalence.

b) We can suppose now, without loss of generality, that $f$ is an element of our weak equivalence class as the one just  constructed in a).
Consider at each positive number $x$ the affine functions 
$A_x(t)\,=\,f(x)+f(x) \frac{t}{x}$
One can verify that given the properties of $f$ that
$\,\forall\,t\,,f(t)\,\leq\,A_x(t)$.More over 
$A_x(x)\,=\,2f(x)$ Consider now the infimum $\,h\,$ of all concave functions bigger or equal to $\,f\,$ .

We have on every point x 
$$\,f(x)\,\leq\,h(x)\,\leq\,A_x(x)\,=\,2f(x)$$
thus $\,h\,$ is weakly equivalent to $\,f\,$ and hence satisfies all the properties to belong to 
$\,\mathcal{W}.$

\end{proof}

\begin{remark}
The family of weak equivalence classes is uncountable 
because each of the functions $x^a\,$ with $\,a\in\,(0,1)$is a weight function belonging to $\,\mathcal{W}\,$,and for two different exponents those functions belong to different weak equivalence classes.

\end{remark}

Let us now investigate some properties of the strong inequality :

\begin{theorem}
The strong inequality $(\,<<\,) $ ,defined in the space of weight functions
 $\,\mathcal{W}\,$ satisfies the following properties:
 
 a)For any $\,\omega\,\in\mathcal{W}$, there exists  $\omega'\in\mathcal{W}\, $ such that $\,\omega\,<<\,\omega'$.
 
 b)For any $\,\omega\,\in\,\mathcal{W}$ there exists 
 $\,\omega_1\,$ such that $\,\omega_1\,<<\,\omega\,$
 
 c)For any couple $\,(\omega_1,\omega_2)\,$
 such that $\,\omega_1 \,<<\,\omega_2\,$ there exists 
 $\,\omega_3\,\in \,\mathcal{W}\,$ such that

 $$\,\omega_1\,<<\,\omega_3\,<<\,\omega_2\,.$$
 
 d)For any couple $\,(\omega_1,\omega_2)\,$, there exists
 $\,\omega_3\,\in \,\mathcal{W}\,$, such that 
 $\,\omega_1\,<<\,\omega_3\,$ 
 and $\,\omega_2\,<<\,\omega_3\,.$
 \end{theorem}

\begin{remark}
Let $\omega_1$ be weakly  equivalent to $\omega_3,$ and 
$\omega_2 \,$ be weakly equivalent to $\omega_4\,$ then it is easy to see that if $\omega_1\,<<\,\omega_2
\,$ then we also have $\omega_3\,<<\,\omega_4\,,$
Thus the strong inequality depends only on the weak equivalence class and thus defines a "strong inequality "
on the family of equivalence classes which becomes thus a directed set
\end{remark}

\begin{proof}

In all those proofs we will suppose from now on, without loss of generality, that our fonctions $\, \omega(x)\,$ are concave or  such that $\,\phi(x)\,=\,\frac{\omega(x)}{x}\,$ is decreasing;

a)Let $\,A\,=\,\int_1^\infty \frac{\omega(t)}{t^2}dt\,$
We can construct inductively an increasing to infinity sequence $\,(T_n)\,$ such that
$$\int_{T_n}^{T_{n+1}}\frac{\omega(t)}{t^2}dt\,
\leq\,\int_{T_n}^\infty \frac{\omega(t)}{t^2}d\,
\leq\,\frac{A}{2^n},.$$.

 Consider now, on each interval
 $\,(T_n,T_{n+1})\,$ the affine function 
$\,A_n$ taking on $\,T_n\,$ the 
value $\,n\omega(T_n)\,$ and on $\,T_{n+1}\,$ the value
$\,(n+1)\omega(T_{n+1})\,$.
Clearly the function $\,A_n\,$ is increasing
We can impose to  $\,T_{n+1}\,$ to be far enough from 
$\,T_n\,$ in such a way that the function which on each of those ontervals is equal to
$\,\phi_n(t)\,=\,\frac{A_n(t)}{t }\,$ is decreasing .
 Consider now the function $\,\omega'\,$
 such that on any of those intervals $\,(T_n,T_{n+1})$ it takes the value of the maximum of
  $\,n\omega\,$ and $\,A_n\,$. By concavity of$\, \omega\,$, and $\, \omega'\,\leq\, (n+1)\omega\,$ on $\,(T_n,T_{n+1})$. One  can easily verify that this function is increasing and by good choice of the $T_n$ ,
   $\,\frac{\omega'(x)}{x}\,$ is decreasing. Morover it satisfies all properties in order to belong to our space of weight functions, because it is increasing to infinity and
   $\,\frac{\omega'}{t}\, $ is decreasing to zero,
   it is bigger than $\,\omega\,$ Thus
   
   $\lim(\frac{\omega(t)}{log(1+t)})\,=\,\infty.$
    and  
   $\int_{T_n}^{T_{n+1}}\frac{\omega'(t)}{t^2}dt \,\leq\,
   \frac{(n+1)A}{2^n}\,$ which implies that 
   $\int_1^\infty\frac{\omega'(t)}{t^2}dt \,<\,\infty\,$ 
    
   Clearly by construction $\omega(t)\,= o(\omega'(t))\,$ 
thus,$$\,\omega\,<<\,\omega'\,$$

b) As $\,\frac{\omega(t)}{log(1+t}\,\rightarrow\,\infty\,$
 we can construct inductively a sequence $\,(T_n)\,$ increasing to infinity
 such that that the following properties are satisfied
 $$a)\,\forall\,t\,\geq\,T_n\,
 ;\frac{\omega(t)}{log(1+t)}\,\geq\,n^2. $$
Moreover if we consider on every interval
 $\,(T_n,T_{n+1})\,$ the affine function $\,B_n\,$ which on 
 $\,T_n\,$ takes the value $\,\frac{1}{n}\omega(T_{n})$  
and on $\,T_{n+1}\,$ the value
 $\,\frac{1}{n+1}\omega(T_{n+1})$,then we can construct inductively this sequence $\,(T_n)\,$ in such a way that   $\,B_n\,$ is increasing and $\,\frac{B_n(t)}{t}\,$ is decreasing .
 Consider now the function $\,\omega"\,$,which on any interval $\,(T_n,T_{n+1})\,$ is equal to the maximum of
 $\,\frac{1}{n+1}\omega(t)\, $ and $\,B_n$ .This function is clearly increasing and when divided by $\,t\,$ the result is decreasing .More over as it is smaller than $\,\omega\,$ and for
  $\,t\,\geq\,T_n\,$ we have: 
  $\frac{\omega"(t)}{log(1+t)} \,\geq\,\frac{n^2}{n+1}\,$. We can easily conclude that it satisfies all conditions to belong  to our family of weight functions $\,\mathcal{W}\,.$
  By construction it is clear that it is $\,o(\omega(t)\,$
 Thus $\,\omega_1\,<<\,\omega\,.$
 
 c)Let us consider the function
  $\,\omega_3(t)\,= \,\sqrt{(\omega_1(t)\omega_2(t))}\,$
  It is immediate to verify that this functions verifies all the requirements and that 
  $\,\omega_1\,<<\omega_3\,<<\,\omega_2\,$
 
d)Consider the function
 $\,\omega"\,=\,sup(\omega_1,\omega_2)\,$ It is easy to verify that 
  it belongs to $\,\mathcal{W} \,$.
  Let $\,\omega_3\,$ be an element of $\,\mathcal{W}\,$
  such that $\,\omega"\,<<\,\omega_3\,.$
  We have both $\,\omega_1\,<<\,\omega_3$
  and $\,\omega_2\,<<\,\omega_3\,.$

\end{proof}

Now we establish a lemma which will be important for the rest of this paper. 

(Analogous results are present in the litterature about weight fuctions and ultradifferentiable functions but we give a proof for self consistency of the paper (an analogous result was also prooved for the sequence presentation of ultradiferentiaition in 
\cite{D.V.V.} )
\\
`
\begin{lemma}

  FUNDAMENTAL LEMMA
  \\
  
Let $\,\omega\,$ be a weight function belonging to
  $\,\mathcal{W} $, then the following results hold:
  \\
  
a)A continuous function
 $\,g(t)\,:\,\RR^+\,\rightarrow\,\RR^+\,$ is
  $o(e^{k\omega(t)})$ for all strictly positive constants $ \,k\,$ if and only if there exists
 $\,\omega'\,\in\,\mathcal{W}\,,\omega'\,<<\,\omega\,$ 
  such that $\,g(t)\,=\,o(e^{\omega'(t)})$
\\  
  b) A continuous function 
   $\,h(t)\,:\,\RR^+\rightarrow\,\RR^+\,$ is 
$o(e^{-l\omega(t)})$ for some positive $\,l\,$ if and only if 
$$\,\forall\,\omega'\,<<\,\omega\,,
 h(t)\,=\,o(e^{-\omega'(t)})\,$$

\end{lemma}

\begin{proof}
a) The property: 
$$ \,\forall\,k\,>\,0\,, g(t)\,=\,o(e^{k\omega(t)})\,$$
implies that for all integer n we can find $\,T_n\, ,$
as large as we want, such that 
$$\,\forall \,t\,\geq\,T_n\,,\,g(t)\,
<\,\frac{1}{n}e^{\frac{1}{n}\omega(t)}\,.$$
We can thus construct inductively an increasing to infinity sequence $\,(T_n)\,$ with each 
$\,T_n\,$ satisfyning the above property as well as all the properties of the proof of point b) of theorem (1.4) for the construction of $\,\omega''\,<<,\,\omega\,.$
Clearly this $\,\omega'' \,$ thus constructed satisfies
the property :
$\,g(t)\,=\,o(e^{\omega''(t)})\,.$
Conversely , notice that if $\,\omega_1\,<<\,\omega\,,$
for any $\,k\,>\,0\,,$ for $t$ large enough
 $ \,\omega_1\,<\,k\omega(t)\,$ and for any $\,t\,$ large enough we have  $\,\omega'(t)\,<\,k\omega(t)$ ,Thus if for some $\omega_1 $ we have 
 $$ h(t) \,=\,o(e^{\omega_1(t)})$$   we also have 
 for any $k>0$ :
 $$ h(t) \,=\,o(e^{k\omega(t)}).$$
 b) Let us now suppose that for some positive $l$ we have
 $ \,h(t) \,=\,o(e^{-l\omega(t)})$ then as for any 
 $\,\omega_1 <<\,\omega\,,$ for $t$ large enough we have 
 $\,\omega_1(t) <l\omega(t)\,$ we can easily conclude that 
 $$ \,\forall \omega_1\,<<\,\omega\,$$
 $$\, h(t) \,=\,o(e^{-\omega_1(t)})\,$$

 Conversely let us suppose that 
 $$ \,\forall\,\omega'\,<<\,\omega\,
 ,h(t)\,=\,o(e^{-\omega'(t)})\,$$ 
but there does not exist any $\,l\,>\,0\,$ such that 
$\,h(t)\,=\,o(e^{-l\omega(t)})\,.$
We will prove that in this case this contradicts the previous hypothesis;
If this was true for any integer $\,n\,$, for any $\,M\,>\,0$ and $\,B\,>\,0$ there would exist $\,T\,>\,M\,$

such that 
$\,e^{\frac{1}{n}\omega(T)}h(T)\,>\,B\,.$
Thus we can construct inductively an increasing to infinity sequence $\,(T_n)\,$  satisfying all the conditions of the proof of point b) of theorem (1,4) and such that $\,e^{\frac{1}{n}\omega(T_n)}h(T_n)\,>\,n\,$. 
Thus if $\,\omega"\,$ is the weight function constructed as in this proof by basing the construction on this sequence
 $\,(T_n)\,$,
  we can see that
 $e^{\omega"(T_n)}h(T_n) \,\rightarrow\,\infty$
 which contradicts the hypothesis on $\,h\,$.

\end{proof}

\section{"weight functions based" ultradifferentiable functions}

Here we follow the presentation of the subject in \cite{D.V.V.};we use the definitions and choice of norms used in
 \cite{D.V.V.}. and use their properties presented in this paper.
 
  Let $\,\omega\,$ be  a fixed weight function belonging to $\,\mathcal{W}\,.$ In \cite{D.V.V.} was presented the fact that the following families of seminorms defined on test fuctions with supports in a given compact subset $\,K\,$:

  a)$$ ||f||_{K,l}^{\omega} \,
  =\,\int |\hat{f}(\xi)|e^{l\omega(\xi)}d\xi .$$
  b)$$||f||_{K,l}^{\omega,\infty}\,
  =\,sup(\hat{f}(\xi)|e^{l\omega(\xi)})$$
  c) $$||f||_{K,l}^{\omega,2} \,=\,
  (\int|\hat{f}(\xi)|^2e^{2l\omega(\xi)}d\xi)^{\frac{1}{2}} .$$
  Those families of seminorms
 satisfy the following relations :There exist positive constants 
 $\,C_1,C_2,C_3,C_4\,$ such that:
 $$C_1||f||_{K,l}^{\omega,\infty} \,\leq\,||f||_{K,l}^{\omega}\,\leq\,C_2||f||_{K,2l}^{\omega,\infty} \,$$ and
 $$C_3||f||_{K,l}^{\omega,2}\,\leq \,
 ||f||_{K,l}^{\omega} \,\leq\,C_4||f||_{K,2l}^{\omega,2}\,.$$
 (when it is clear which norms we use we somtimes will skip the suplementary indexes in notations ).
 
 Thus the topological considerations can use any of those families of seminorms. 
 In this paper we will mainly use the first one. But in the last section we will uses mainly the third one 
 We will now remind the definitions of Beurling and Roumieu ultradifferentiable functions with our given weight function$\,\omega,$. We start with the test functions :
 \begin{definition}
a)Beurling case :
$ \,f\,$ belongs to $ \mathcal{D}^{(\omega)}\,$ if it has compact support K and 
$\,\forall l>0\,, ||f||_{K,l}^\omega \,< \,\infty.$
 
b)Roumieu case  :
$ \,f\,$ belongs to $ \,\mathcal{D}^{\{\omega\}}\,$ if it has compact support,$\,K\,$ and 
$$\,\exists l>0\,,\mbox{s.t.}\,
||f||_{K,l}^\omega \,<\,\infty \,$$

 The duals of the above spaces are the corresponding ultradistribution spaces.
 The spaces of the ultradifferentiable functions of those types ($\,\mathcal{E}^{(\omega)}\,$ and
  $\,\mathcal{E}^{\{\omega\}}\,$ respectively,) are the smooth functions $\,f\,$, such that, whenever 
 $\,\phi\,$  is a test function of the corresponding type, the product $\,\phi f\,$ is a test function of this type.
 
 Given a compact set K we define the norm 
 $\,\||f||_{K,l}^\omega\,$ as the minimum of 
 all $\||\phi f||_{K,l}^\omega $ where $\,\phi\,$ is a test fuction of the corresponding type which takes the value $1$ on the compact $\,K\,.$ (the same considerations hold for the other choices of seminorms proposed before)
 
 The duals of those spaces of test functions are the corresponding ultradistribution spaces.
 \end{definition}
 \begin{remark}
 Using the properties of the above seminorms one can see easily that the definition of those spaces does not depend on the choice of weight function inside a given weak equivalence class
 \end{remark}
 
 Our first theorem on projective presentations is the following:
 
 \begin{theorem}
 Consider a fixed weight function
 $\,\omega\,\in\,\mathcal{W}\,$.
  Then a function $\,\phi\,$ belongs to 
  $\,\mathcal{D}^{\{\omega\}}(\Omega])\,$ , if and only if 
  $$\,\forall\, K\,\subset\subset\,\Omega\,
  ,\forall\,\omega'\,<<\,\omega\,,
  ||\phi||_{K,1}^{\omega'}\,<\,\infty\,.$$
 
 \end{theorem}
 
 \begin{proof}
 
 This result is also explicit in \cite{D.P.V.} (corolary 2) and probably also in explcit or implicite form in other papers ,  but we present also our proof for self consistency of the paper.
 As before we can suppose , without loss of generality, that $\,\omega\,$ is increasing to infinity, but 
 $\,\frac{\omega(t)}{t}\,$ is decreasing .
 Notice that if $\,\omega'\,<<\,\omega\,$ , for all $\,l\,>\,0\,$ , for $\,t\,$ large enough 
 $\,\omega'(t)\,\leq\, l\omega(t)$.
 Thus if
  $\,\phi\,\in\,\,\mathcal{D}^{\{\omega\}}(\Omega)\,$ i.e.
  $\,\exists\,l\,>\,0\,\mbox{s.t.}\,
 \,\int|\hat{\phi}(\xi)|e^{l\omega(|\xi|)}d\xi\,<\,\infty$
 
  then we also have
   $\,\int|\hat{\phi}(\xi)|e^{\omega'(|\xi|)}d\xi\,<\,\infty\,.$
  
    Conversely suppose that 
    $$\,\forall\,\omega'\,<<\,\omega\,
    ,\int|\hat{\phi}(\xi)|e^{\omega'(|\xi|)}d\xi
    \,<\,\infty\,.$$ 
    If it did not belong to
     $\,\mathcal{D}^{\{\omega\}}(\Omega)\,$
    this would imply that for any integer $\,n\,$; 
    
$\,
\int|\hat{\phi}(\xi)|e^{\frac{1}{n}\omega(|\xi|)}d\xi\,=\,\infty\,.$
    This would imply that we can construct inductively an increasing sequence $\,(T_n)\,$  such that  
$$\,\int_{|\xi|<T_n}|\hat{\phi}(\xi)|e^{\frac{1}{n}\omega(|\xi|)}d\xi\,>\,n\,,$$ and we can impose inductively in the construction of this sequence  the fact  that it satisfies all the conditions of the construction of 
$\,\omega"\,$
 in the proof of theorem 1.4 point b) of the existence of 
 $\,\omega"\,<<\,\omega\,.$
 With this function $\,\omega"\,$ thus constructed with the help of our sequence $\,(T_n)\,$
 we clearly have: 
 $\,\int|\hat{\phi}(\xi)|e^{\omega"(|\xi|)}d\xi
 \,=\,\infty\,$
 which contadicts the hypothesis. 
 Thus the two conditions are equivalent.

 \end{proof}

 \section{asymptotic scales and "generalisations" of topological vector spaces}
 
Most constructions of algebras following the ideas of Colombeau 
generalized functions are implicitely or explicitely  based on 
the notion of "asymptotic scales". 
 This notion was expicitely presented in \cite{D.S.1}, 
 \cite{D.S.2} and an analogous notion is implied in 
 \cite{D.H.P.V.}
 the constructions thus obtained are special cases 
 oftheconstructions of $(\mathcal{C},\mathcal{E},\mathcal{P})$-
 algebras introduced by A. Marti (\cite{M.1}
 and we give here a slightely more general  version of this 
 notion.   This notion will be used here with two different 
 assymptotic scales for the construction of algebras of 
 generalized functions as was also done in the "sequence model" 
 of ultradifferentiability in\cite{D.V.V.}.
 
 \begin{definition}
An asymptotic scale is  a family $ \mathcal{A}\,$ 
  of continuous strictly decreasing to zero functions from 
 $\,(o,1)\,$ to $\,\RR^+,$ ,indexed by a directed set ,$\,B\,$, directed by  a relation that we note 
 $\,<< \,$
  satisfying the two following conditions :
 
 $$\forall(l,m) \in \, \mathcal{B} \times\mathcal{B} , 
 m<<l\,, \rightarrow \, a_l(\varepsilon) = o(a_m(\varepsilon))$$
 
 $$ \forall(l,m) \in \, \mathcal{B} \times \mathcal{B} \, ,\exists
 \,k\in \mathcal{B},      \mbox{s.t.} \, a_k(\varepsilon) =  o(a_l(\varepsilon)a_m(\varepsilon))\,.$$

 \end{definition}
 Examples :
 The family $\{\varepsilon^n,n\in\NN \},$ is the first asymptotic scale considered in the simplified definition of Colombeau generalized functions.
 One can also easily verify that given a fixed weight function $\,\omega\,\in \,\mathcal{W}\,$ the families
  $\{e^{-k\omega(\frac{1}{\varepsilon})}\,, k\,>\,0\,\}$                    
 and
 $\{e^{-\omega'(\frac{1}{\varepsilon})}\
 ,\omega'<<\omega \}$ are asymptotic scales the first indexed by positive real numbers and the second by the weak equivalence classes set on which the strong inequality gives a relation that makes it a directed set.
 
 We will now present some basic notions in this frame and sketch how one can build "generalizations" of a topological metrisable  vector space using an asymptotic scale
  (with a "projective" definition)

 \begin{definition}

A strictly positive function $g: \RR^+ \,\rightarrow \RR^+$
is said to be  $ \mathcal{A}$-moderate if 
  $$\forall\,l \in \mathcal{B}\,,\mbox{large enough}\, ,    \exists m(l) \in\mathcal{B}\,,
  \mbox{such that} :\, g(\frac{1}{a_l(\varepsilon)}) =
   o(\frac{1}{a_{m(l)}(\varepsilon)}) ,$$
   It is said to be $ \mathcal{A}$-compatible if 
 there exists an increasing mapping
  $ u:\mathcal{B} 
\rightarrow \mathcal{B}$ s.t.the following properties hold
$$ \forall k\in \mathcal{B}\,, \exists l\in \mathcal{B}\,
 \mbox{s.t.},
u(l) >> k$$
$$ \forall k\in \mathcal{B}\,, g(a_k(\varepsilon)\,=
\,o(a_{u(k)}(\varepsilon))\,
  $$
\end{definition}

We now show the usual way to generalize a metrisable vector space with the help of a fixed given scale
 $\,(\mathcal{A},B)\,$ . Consider a metrisable vector space
 $ \,E\,$ topologized by a family $\,(\mu_i)\,i\in\,J$ the presentation being made with nets and the parameter
  $\varepsilon \\in\,(0,1)$ going to zero .
 
\begin{definition}

The family of $(\mathcal{A})-$moderate nets of $\,E\,$ is defined as :
\begin{equation*}
\mathcal{E}^{(\mathcal{A})}_M [E] = 
\left\{(g_{\varepsilon}) \in E^{(0,1)}\,:\, \forall \,
i\in I, \exists\,a\,\in\,\mathcal{A}\,,\mbox{s.t.}\,
\mu_i(g_\varepsilon)\,=\,
    o(a(\varepsilon)^{-1}).
    \right\}
\end{equation*} 
 While the family of negligible nets is defined as:
\begin{equation*}
\mathcal{N}^{(\mathcal{A})} [E] = \left\{(g_{\varepsilon}) \in E^{(0,1)}\,:\, \forall \,i\in J\, , \forall\,a\,\in\,\mathcal{A}\,,
\mu_i(g_\varepsilon)\,
=\,o((a(\varepsilon))
    \right\}.
\end{equation*} 

The generalized $(\mathcal{A})$-extension of $\,E\,$
is the factor space

$$\,\mathcal{G}^{\mathcal{A}}[E]\,=\,
\mathcal{E}^{(\mathcal{A})}_M [E]/\mathcal{N}^{(\mathcal{A})} [E]. $$
 
In this type of constructions there is a natural way of defining 
a topology, called sharp topology,analogous to what was done in 

usual Colombeau algebra (\cite{S.1},\cite{S.2},\cite{S.4}
 \cite{G.1}, \cite{G.2} This type of constructions was generaliszed in the case of other scales ( \cite{D.S.1}
 ,\cite{D.S.2}
 In our case we obtain a sharp topology by considering a basis  
of neighbourhoods of zero of the form
$$\,B(i,a)\,=\,\{(g_{\varepsilon)}\,\in E^{(0,1)}\,\mbox{s.t.}
\mu_i(g_\varepsilon)\,=\,o(a(\varepsilon)) \}$$
$a$ being an element of our scale. 
 One can easily verify that as the intersecion of all those "boxes" is the space of negligible nets this topology passes to the factor space and defines a hausdorf topology.

\end{definition}

  This kind of ideas can be extended on every topological vector space with  more general definitions  but in this paper we will mainly work with previous formulations . 
  
  \begin{definition}
 Let E be a topological vector space with a basis of neighbourhoods of zero $\mathcal{B}=\{B_i, i\in I \} $ we say that a net $\, (g_\varepsilon )\,$ is moderate if 
  
   $$ \forall i\in I , \exists a \in \mathcal{A} \mbox{s.t.}
  a(\varepsilon) g_\varepsilon  \in B_i$$ for $\varepsilon$ small enough 
  and we will say that it is negligible if 
  
  $$ \forall i\in I , \forall a \in \mathcal{A}\, 
  a(\varepsilon)^{-1}g_\varepsilon \,\in B_i$$ for 
  $\varepsilon$ smal enough .the $\mathcal{A}$ colombeau extension $\, \mathcal{G}^\mathcal{A} [E] \,$ is the factor space of the space of moderate nets by the space of negligible nets. 
   In similar way as befor we define the sharp topology 
   \end{definition}
  In this presentation it is also possible to define sharp topologies in an analogous way to what was done previously . 
   This is straughtforward but stays outside the scope of this paper

This construction has a "functorial" aspect and if 
$$ \L\,:\,E\,\rightarrow\,F$$ is a continuous linear mapping between vector spaces topologized by families of seminorms then there is a natural extension of $\,L\,$ between the respective generalizations,continuous for the respective sharp topologies.
  But there is a larger class of mappings that can pass through the generalisation process:
  \begin{definition}
A mapping $\,\Phi\,$ from a metrisable vector space $\,E\,$, topologized by a family $(\mu_i),i\in I$ of seminorms 
to a another metrisable vector space $\,F\,$ topologized by 
a family $(\nu_j),j\in J$ of seminorms is said to be
  $\mathcal{A}$-continuously tempered if 
for any seminorm $\nu_j$ among those which are used to topologize F,there exists a finite set $\,A\;\subset I, $  a 
$\mathcal{A}$-moderate function $ g $ and a
 $\mathcal{A}$-compatible function $h$ such that if for a finite subset $A$ of indices we  note
  $\,\mu_A \,=\,sup\{\mu_i,i\in A \},$
   then 
$\,\nu_i(\Phi(x)) \,
\leq\,g(\mu_A(x))\, $ and if 
$\,\mu_A(h)<1$ then
 $$\,|\nu(\Phi(x+h)-\Phi(x)|\,<\,g(\mu_A(x))h(\mu_A(l))\,.$$
\end{definition}

One can verify that: (see \cite{D.S.1}) .
\begin{proposition}

If $\,\Phi\,$ is a $\mathcal{A}$-continuously tempered mapping between the topological vector space $\,E\,$ topologized by the family $\,\{\mu_i,i\in I\,\}$ of seminorms to the topological vector space $\,F\,$ topologized by the family 
$\,\{\nu_j,j\in J\,\}$ of seminorms, then it defines 
 a mapping between $\mathcal{G}^{\mathcal{A}}[E]$
 to $\mathcal{G}^{\mathcal{A}}[F]$ which is continuous for the respective sharp topologies.
\end{proposition}
the slight generalisation of our definition does not prevent the same proof to apply.Direct proof is straight forward (but fastidious)

 (this "functoriality" is usefull when we want to investigate "well posedness" in the case of problems modelized non linear equations)

\begin{remark}
Here we exposed the  " projective" construction but there is an alternative "inductive" strategy which will be used also in the sequel and we will present here for simplicity  only 
in the case of dimension one !
\end{remark}

 \begin{definition}
A net $\,(r_\varepsilon)$ of constants is said to be (inductively) moderate for the asymptotic  scale $\,\mathcal{A}\,$ 
if for all elements   $\,a(\varepsilon)\,$ of our scale
 
 $$\,|r_\varepsilon| \,=\,o(a(\varepsilon)^{-1})\,.$$
 It is said to be(inductively) negligible if there exists an element of the scale $\,a(\varepsilon)\,$
  such that
  $$\, |r_\varepsilon| \,=\,o(a(\varepsilon))\,  .$$
  
  The "inductively" generalized constants are now the factor space of "moderates" by the space of "negligibles".
    \end{definition}
 This type of "inductive" definitions that wll be used to embed Roumieu ultradistributions into an algebra do not give in a natural way a "sharp" topology as in the case of the "projective" constructions and for this reason we will have to find a "projective" presentationnwhenever this is possible !
 
  \section{Beurling-Bj\"{o}rk and Roumieu- Bj\"ork  generalized ultradifferentiable functions}

\subsection{Beurling case}
A) BEURLING CASE
We will now remind how was defined an algebra in which were embedded the Beurling-Bjork ultradistributions
(\cite{D.V.V.}) in the frame of a fixed given weight function
 $\omega\in\mathcal{W}$

\begin{definition}

The moderate nets of Beurling-$(\omega)$ ultradifferentiable functions are defined as the elements of $\mathcal{E}^{(\omega)}_M(\Omega)$ where:

\begin{equation*}
\mathcal{E}^{(\omega)}_M(\Omega) = \left\{(g_{\varepsilon}) \in \mathcal{E}^{(0,1)}_{(M_p)}\,:\,  \forall K \subset \subset \Omega,\,\forall h>0\,,\exists k>0 , \quad
    \| g_{\varepsilon}\|_{K,h}^{\omega} 
    =o(e^{k\omega(\frac{1}{\varepsilon})})\right\}
\end{equation*}  
while the negligible nets are  those belonging to 
The space $ \mathcal{N}^{(M_p)}(\Omega) $  defined by:
\begin{equation*}
\mathcal{N}^{(\omega)}(\Omega) = \left\{(g_{\varepsilon}) \in
 (\mathcal{E}^{(\omega)})^{(0,1)}
 \,:\, \forall K \subset \subset \Omega\,,
 \forall h>0, \forall k >0 \quad
\|g_{\varepsilon}|_ {K,h}^{\omega} =
 o(e^{-k\omega(\frac{1}{\varepsilon})}) \,\right\}
 \end{equation*}  

The factor space 
$$\mathcal{G}^{(\omega)}(\Omega) \,=\,
\mathcal{E}^{(\omega)}_M(\Omega)/\mathcal{N}^{(\omega)}(\Omega)$$
is the algebra of Beurling Bjork generalized ulradifferentiable functions.
The class of a net $(g_\varepsilon)$ will be noted 
$[(g_\varepsilon)]$

\end{definition}

Clearly this is the extension of the topological vector
space of Beurling Bjork $(\omega)$- ultradifferentiable functions using the assymptotic scale 
$$\,\{e^{-k\omega(\frac{1}{\varepsilon})}\,,
 k\in\RR^+\,\}$$
Let $(\Omega_n)$ be a countable exhausting sequence of relatively compact open subsets of $\Omega$ ,i.e 
 such that their union is $\Omega$  and
 $\,\Omega_n\,\subset\subset\,\Omega_{n+1}\,$
A countable basis for its Haudsdorff "sharp" topology can be defined considering the  family $\,B_{(n,m,l)}\,$
where
$$ \,B_{(n,m,l)}\,
=\,\{ [(g_\varepsilon)] \,\mbox{such that}\,
|g_\varepsilon|_{\Omega_n,m}^{(\omega)}\,
=\,o(e^{-l\omega(\frac{1}{\varepsilon})})\,,$$
and choosing a countable adequate collection of positive numbers  $m$ and $n$ .
As the intersection of those is the space of negligible nets this topology passes to our algebra which is the factor space, and thus defines a Hausdorff topology on this factor space.
 We can in this way define also the rings
 $$ \tilde{\CC}^{(\omega)}$$  and 
$$ \tilde{\RR}^{(\omega)}$$ 
 of Beurling(repsectively complex and real generalized constants and .
Likewise if we consider $\RR^d$ we have the Beurling generalized points of this space etc.

  In a more general way, if we replace for a topological space $E$ topologized by a family $(\mu_i, i\in I)$, of seminorms in the above definition the norms $||g_{\varepsilon}||_{K,h}$ by the seminorms  $\mu_i (g_{\varepsilon}|)$ we establish in the same way a "sharp" Hausdorff topology on the $(\omega)-$Beurling-Colombeau generalization of $E.$ As previously, if there is a countable sequence of seminorms defining the uniform structure and the topology of $E$ then there is a countable basis of neighbourhoods of zero. 

%****

We have the following theorem:
  
\begin{proposition}
If $E,(\mu_i ,i\in I)$ is a topological vector space topologized by the family  $(\mu_i ,i\in I)$ of seminorms then its $(\omega)$-Beurling Colombeau generalization $\mathcal{G}^{(\omega)}[E]$ is a topological module over the ring of generalized constants.  If $E$ is a topological algebra then its generalisation  is also a topological algebra  over  the corresponding ring of generalized constants and thus $\mathcal{G}^{(\omega)} (\Omega)$ is a differential topological algebra over the ring   $\tilde{\mathbb{C}}^{(\omega)}$ of 
$\, (\omega)-$ Beurling generalized constants.
\end{proposition}

\begin{proof}
As our construction is a special case of $(\mathcal{C,E,P})$ algebra (see \cite{M.1}, \cite{D}, \cite{D.S.1} ) all linear and bilinear mappings can be naturally extended as continuoys mappings for the respective sharp topologies ,also into the new setting 
Thus we have topological modules and  a TOPOLOGICAL DIFFERENTIAL ALGEBRA for the sharp topology presented previously .but this can be also verified directly using the definitions..
\end{proof}
 as linear continuous mappings are a special case of continuous tempered mapping (or by easy direct computations) one can  verify that:

\begin{proposition}
If $P(D)$ is a $(\omega)$ ultradifferentiable linear operator it defines a continuous linear mappnig for the above defined sharp topologies of our algebra and  of the corresponding  ring of  generalized constants.
\end{proposition}

\begin{proof}
Such an operator is linear and continuous for the topology of the basic Beurling space thus it extends to a continuous mapping as an application of the general property on generalisations of continuous linear mappings  via any given assymptotic scale.

\end{proof}
 \begin{remark} ..When the solution operator of a linear problem is also continuous it gives birth to a continuous linear operator in the new frame. Thus the solution operator , considered this time in the "generalized " frame, is also continuous and the problem is "well posed" also  in the generalized frame.
\end{remark}

\subsection{Roumieux case}  
                            
Let us now remind the construction of an algebra in the Roumieu case. In some sense it looks like the Beurling case but with a change of quantificators (it is an "inductive" construction).
\begin{definition}

The space of $\{\omega\}$ moderate nets of Roumieu ultradifferentiable functions is:
   \begin{equation*}
\mathcal{E}^{\{\omega\}}_M(\Omega) = \left\{(g_{\varepsilon}) \in \mathcal{E}^{{\{\omega\}}^{(0,1)}}\,:\, \forall K \subset \subset \Omega,\forall k>0\,, \exists h>0,  \mbox{s.t.}\,
    \| g_{\varepsilon}\|_{K,h}^{\omega} =o(e^{k\omega(\frac{1}{\varepsilon})})\right\} .
\end{equation*}   
The space of negligible ones is :

\begin{equation*}
\mathcal{N}^{\{\omega\}}(\Omega) =
 \left\{(g_{\varepsilon}) \in\,
  \mathcal{E}^{(0,1)}_{(M_p)}
 \,:\, 
 \forall K \subset \subset \Omega, \exists k>0\,\exists h>0 ,\quad
\|g_{\varepsilon}|_ {K,h}^{\omega} =
 o(e^{-k\omega(\frac{1}{\varepsilon})}) \,\right\},
 \end{equation*}
 
As previously the space of Roumieu generalized ultradifferentiable functions is the factor space of "moderates" by "negligibles" .

\end{definition}

\begin{remark}
An analogous shift of quantificators must be done for the "Roumieu extension " of any metrisable space. Thus the 
 Roumieux generalized constants
 $\,\tilde{\CC}^{\{\omega\}}\,$  will be the factor space of the space of "Roumieu moderate nets of constants , i.e. those that are
  $\,o(e^{m\omega (\frac{1}{\varepsilon})})\,$ for all
  $\,m>0\,$ 
  by the space of "Roumieu negligible nets ; i.e. those 
  that are
  $\,o(e^{-m\omega (\frac{1}{\varepsilon})})\,$ for some
  $\,m>0\,$ .
  
  \end{remark}
  
  In order to apply in the Roumieu case the ideas of sharp topologies  we have to use a  projective description of the "Roumieu Bj\"{o}rk" construction,. analogous to the one constructed for the construction of the Roumieu "$\,(M_p)\,$ case in  \cite{D.V.V.}.
  
  We will prove that the above definition of Roumieu Bj\"{o}rk algebra is equivalent to the following "projective" definition:
  \begin{definition}
 Let 
$$  \tilde{\mathcal{E}}^{\{\omega\}} (\Omega) = 
\{(g_{\varepsilon}) \in (\mathcal{E}^{\{\omega\}})^{0,1)} \mbox{s.t.} \forall \omega_1 <<\omega\,,
 \forall K \subset \subset \Omega, \exists \omega_2 <<\omega\, \mbox{s.t.}\,
  ||g_{\varepsilon}||_ {K,\omega_1}^1\, =
 \, o(e^{\omega_2(\frac{1}{\varepsilon})}) \}$$.

Likewise the set of negligible nets is defined by :
  
$$  \tilde{\mathcal{N}}^{\{\omega \}} (\Omega) =
  \{(g_{\varepsilon}) \in 
  (\mathcal{E}^{\{\omega\}})^{(0,1)} \mbox{s.t.},
 \forall K \subset \subset \Omega , \forall \omega_1<<
 \omega,
 \forall \omega_2 <<\omega\,,
 ||g_{\varepsilon}||_ {K,1}^{\omega_1}\,=
 \, o(e^{-\omega_2(\frac{1}{\varepsilon})}) \}\,.$$
 The algebra is now defined as the factor space of those two.
 \end{definition}
 
  The main result of this sectio, analogous to reuslts of the same kind for the sequence model in \cite{D.V.V.} and \cite{P.R.S.Z} is the equivalence of the two definitions:
  \begin{proposition}
 The above two definitions are  equivalent and thus:
   
 $$\,\mathcal{E}^{\{\omega\}}_M(\Omega)\,
 = \,\tilde{\mathcal{E}}^{\{\omega\}}_M (\Omega)\,
 \mbox{and}\,,
 \mathcal{N}^{\{\omega\}}(\Omega)\,
 =\,\tilde{\mathcal{N}}^{\{\omega \}} (\Omega)\,.$$
  
  \end{proposition}

\begin{proof}
   As we are  in the non quasi analytic case we have partition of unity.Thus  we can, without loss of generality, suppose that   all the elements we consider have compact support inside  the interior of some fixed  compact set
   $\,K\subset\subset \Omega\,.$
   Let us suppose that
    $\,(g_\varepsilon)\,\in\,
    \tilde{\mathcal{E}}^{\{\omega\}}_M (\Omega)\,.$
    but it does not belong to 
    $\mathcal{E}^{\{\omega\}}_M(\Omega)\, .$
 One can verify that this would imply that: 
    $$
\,\exists\,k\,>\,0\,\mbox{s.t.}\,\forall\,n\,\in\,\NN\,
||g_\varepsilon||_{K,\frac{1}{n}}^\omega\,\mbox{is not}\,
o(e^{k\omega(\frac{1}{\varepsilon})})\,.$$

Thus there exists $\varepsilon_n $ as small as we want, such that: 

$$
\int|\hat{g}_{\varepsilon_n}(\xi)|e^{\frac{1}{n}\omega(\xi)}d\xi\,>\,2ne^{k\omega(\frac{1}{\varepsilon_n})}
$$
Now we can construct inductively a sequence
 $(\varepsilon_n)$ decreasing to zero and a sequence 
 $(T_n))$ increasing to infinity such that 
 $$\int_{|\xi|<T_n}|\hat{g_{\varepsilon_n}}(\xi)|e^{\frac{1}{n}\omega(\xi)}d\xi\,
 >\,ne^{k\omega(\frac{1}{\varepsilon_n})}\,.$$
 More over we can impose to the construction of the sequence $(T_n)$ all the conditions we used in the proof of the existence of $$ \omega"<<\omega  .$$
  Thus we can construct  such $\omega"$ verifying for every integer $n$ 
 $$ \,\int|\hat{g_{\varepsilon_n}}(\xi)|e^{\omega"(\xi)}d\xi\,>\,\int_{\xi<T_n}|\hat{g_{\varepsilon_n}}(\xi)|e^{\frac{1}{n}\omega(\xi)}d\xi\,>\,ne^{k\omega(\frac{1}{\varepsilon})}.$$
 This implies that
 $||g_\varepsilon||_{K,1}^{\omega"}$ is not
  $o(e^{k\omega(\frac{1}{\varepsilon})})$
  But for any given  $\omega'<<\omega$
   for n large enough (thus $\epsilon_n$ small enough)
   we have
    $e^{\omega'(\frac{1}{\varepsilon})}  < e^{k\omega(\frac{1}{\varepsilon})}\,.$ 
    This would imply that for  $$\forall \,\omega' <<\omega\,,
    ||g_\varepsilon||_{K,1}^{\omega"}\,\mbox{is not}\,
    o(e^{\omega'(\frac{1}{\varepsilon})})$$ 
    which contradicts the hypothesis.
     Conversely let us suppose that 
     $\,(g_\varepsilon)\,$ belongs to
     $ \mathcal{E}^{\{\omega\}}_M(\Omega) .$
     Thus for any $k>0$ there exists $l>0$ 
      such that 
      $$\,
      \int|\hat{g_{\varepsilon_n}}(\xi)e^{l\omega(\xi)}d\xi\,=\,o(e^{k\omega(\frac{1}{\varepsilon})})\,.$$
     
But we can easily see that for any $\omega ' <<\omega, $
 for any given $l>0\,,$ there exists $C_l>0\,$ such that for any positive function
 $$\int \tilde{g_\varepsilon}(\xi)e^{\omega'(\xi)}d\xi\,
 <\,C_l\int\tilde{g_\varepsilon}(\xi)e^{l\omega(\xi)}d\xi \,.$$ Thus 
 is $\,o(e^{k\omega(\frac{1}{\varepsilon})}).$
 As this holds for all $k>0$,
 we can conclude with the help of the fundamental lemma that there exists $ \,\omega"\,$
 such that 
 $$||g_\varepsilon||_{K,1}^{\omega'}\,
 =\,o(e^{\omega"(\frac{1}{\varepsilon})}).$$
 Thus we have proved the equivalence of the two definitions for moderateness.
  Now let us take care of negligibility:
  
  Let us suppose that
$\,(h_\varepsilon)\,\in\,\mathcal{N}^{\{\omega\}}(\Omega).$

This implies that there exist $l>0$ and $\, k>0$ such that 
$$\,||h_\varepsilon||_{K,l}^\omega\,=
\,o(e^{-k\omega(\frac{1}{\varepsilon})}).$$\

Consider some $\omega_1<<\omega\,.$
For given $l>0$ there exists a positive constant $\,C_l\,$
such that $ ||h_\varepsilon||_{K,1}^{\omega_1}\,<
\,C_l||h\varepsilon||_{K,l}^\omega\,=
\,o(e^{-k\omega(\frac{1}{\varepsilon})})$ i.e.
$$
e^{k\omega(\frac{1}{\varepsilon})}||h_\varepsilon||_{K,1}^{\omega_1})\rightarrow \,0.$$

For any $\,\omega"<<\omega$ and for any positive $k>0$, for $\,\varepsilon\,$ small enough we have:
$$\,\omega"(\frac{1}{\varepsilon})\,
<\,k\omega(\frac{1}{\varepsilon}).$$ Thus
$$e^{\omega"(\frac{1}{\varepsilon})}
||h_\varepsilon||_{K,1}^{\omega"}\,\rightarrow  0\,$$
i.e.
$$\,||h_\varepsilon||_{K,1}^{\omega_1}\,
=\,o(e^{-\omega"(\frac{1}{\varepsilon})})$$
This is exactly our second definition of negligibility!
  In order to prove the converse assertion  we will prove that if 
  $\,(h_\varepsilon)\,$ does not belong to 
  $\,\mathcal{N}^{\{\omega\}}(\Omega)$, then it cannot either belong to 
  $\,\tilde{\mathcal{N}}^{\{\omega\}}(\Omega)$
If  $\,(h_\varepsilon)\,$ does not belong to 
  $\,\mathcal{N}^{\{\omega\}}(\Omega)\,,$ this would imply that for any  couple of integgers $\,(m,n)\,$ 
  $\,||h_\varepsilon||_{K,\frac{1}{n}}^{\omega"}\,$ is not
  $\,o(e^{-\frac{1}{m}\omega(\frac{1}{\varepsilon}})\,.$
  This holds even if we chose two equal integgers  and thus $m = n$ This would imply that we can construct inductively a decreasing to zero sequence 
  $\,(\varepsilon_n)\,$ and an increasing to infinity sequence $\, (T_n)\,$ such that :
  $$\,e^{\frac{1}{n}\omega(\frac{1}{\varepsilon_n})}
  \int_{|\xi|<T_n}|\hat{g}(\xi)|e^{\frac{1}{n}\omega(|xi|)}d\xi\,>\,n\,.$$
  Those sequences can be constructed in a way respecting the prescriptions for the construction of $\omega"<<\omega$
  and thus obtain two weight functions 
  $\,\omega_1<<\omega$ and $\omega_2<<\omega$ 
  such that
  $\frac{1}{n}\omega(\frac{1}{\varepsilon}_n)\,=\,
  \omega_1(\frac{1}{\varepsilon}_n)\,$ and
  $ \frac{1}{n}\omega(T_n)\,=\,\omega_2(T_n)\,.$
 Now it is straightforward to conclude that 
 $\,||g_\varepsilon||_{K,1}^{\omega_2}\,$  is not 
$\,o(e^{-\omega_1(\frac{1}{\varepsilon})})\,$
thus $\,(g_\varepsilon)\, $ does not belong to 
$\,\tilde{\mathcal{N}}^{\{\omega\}}(\Omega)\,.$
\end{proof} 
(analogous results for the sequence model were already proved in 
 \cite{D.V.V.}.)

As we have now a "projective " presentation for the construction of our algebra it is easy to define on this algebra a "sharp" Hausdorff topology by defining as before a basis of neighbourhoods of zero: 

\begin{definition}
The Sharp Topology of $\mathcal{G}^{\{\omega\}}(\Omega)\,$
is the topology defined by a basis of neighbourhoods of 0 of the form 
$B_{K,\omega_1,\omega_2}$
where $K\subset\subset \Omega$ and $\omega_1<<\omega $ and 
$\omega_2<<\omega$ are two weight functions belonging to 
$\mathcal{W} \,$ defined,as follows:

$$  B_{K ,\omega_1,\omega_2} = \{g = [(g_{\varepsilon})] \mbox{s.t}  |g_{\varepsilon}||_ {\omega_1} =
 o(e^{-\omega_2(\frac{1}{\varepsilon})})  \}         .  $$ 
\end{definition}

  This construction is analogous to the construction of a sharp topology in the Beurling case but there is a very important difference, because :
   
   \begin{proposition}
 There cannot exist a countable basis for the sharp topologies in Roumieu case 
 \end{proposition}
 
 \begin{proof}
 If such a basis did exist one can see that this would  imply  the existence of a sequence $(\omega_n)$  all stricly smaller than $\,\omega,$ , such  
 $\,\omega_n<<\omega_{n+1} \,$ and such that
 for any $\omega"<<\omega$ there would exist $n$ such that 
 $\,\omega"<<\omega_n\,$
  We can choose without loss of generality, on each weak equivalence class of the elements of the sequence, an element which is increasing and when divided by $x$ becomes a decreaising to zero function and such that 
  $\,\int_1^\infty \frac{\omega_n(t)}{t^2}dt\,
  = \,\frac{1}{2^n}$
  we will now construct a weight function
   $\omega" \in \mathcal{W},\, \omega""<<\omega\,$ 
  such that for any elements $\omega_n$ of our sequence we have
  $\,\omega_n \,<<\,\omega".$
  This will prouve our claim.
  
  We can construct inductively an increasing sequence 
  $\,(T_n)\,$ such that 
  $$\,\forall\,t\,\geq\,T_n\,,\omega_n(t)\,
  \leq\,\omega_{n+1}(t)\,\leq\,\frac{1}{n+1}\omega(t)$$
  On each interval $\,(T_n,T{n+1})\,$ we can consider the affine function $\,A_n(t)\,$ such that
   $\,A_n(T_n)\,=\,\omega_n(T_n)\,$ 
   and $\,A_n(T_{n+1})\,=\,\omega_{n+1}(T_{n+1})$
   We can also impose in the inductive choice of $T_n$ that 
   $\,\frac{A_n(T_{n+1})}{T_{n+1}},\leq\,
   \frac{A_n(T_n)}{T_n} $
   Consider now the continuous function $\omega"\,$
   which on each interval $\,(T_n,T_{n+1})\,$ is defined by $\,\omega"(t)\,=\,\sup (\omega_n(t),A_n(t))\,$
   It is now straightforward to verify that
   $\,\omega"\,\in\,\mathcal{W}\,$ and that 
   $\,\forall\,n\in\NN\,\omega_n<<\omega"<<\omega\,.$ 
   \end{proof}
  \subsection{negligibility criteria}

   Let us now present and prove an important simplification of the characterisation of "negligible nets analogous to what is done in the usual Colombeau case or in ultradiferentiable cases with a definition by sequences \cite{D.V.V.}
   
   \begin{theorem}
 A necessary and sufficient condition for a moderate net 
 $(g_\varepsilon)$ with compact support to be negligible both in  Beurling and Roumieu case is that 
 $$ [(||g_\varepsilon||_2)] \,=\,0$$
 of course this means that even for elements which do not have compact support the negligability of the nets of local $L_2$ norms is enough to conclude to the neglibality of our net 
 \end{theorem}
   
 \begin{proof}
  In both cases we will use thebequivalent family of norms modelled from $L_2$ evaluations 
 1) Beurling case:
 
 By Hypothesis :
 $$\forall l>0\,,\exists\,C_l>0\,, \mbox{s.t.} 
 (\int |g_\varepsilon(x)|^2dx)^{\frac{1}{2}} \,
  <\,C_le^{-l\omega(\frac{1}{_\varepsilon})} .$$
 
 But as our net has compact support this is equivalent to:
 $$\forall l>0\,\exists\,C_l>0\,\mbox{s.t.} 
 (\int |\hat{g}_\varepsilon(|\xi|)|^2 d\xi))^{\frac{1}{2}} \,<\,C_le^{-l\omega(\frac{1}{_\varepsilon})}.$$
  
 Now by h\"{o}lder inequality we can conclude that 
 
 \begin{equation*}
\int e^{k\omega(|\xi|)}|\hat{g}_\varepsilon(\xi)|^2 d\xi) \,
  \leq \,(\int e^{2k\omega(|\xi|)}
  |\hat{g}_\varepsilon(\xi)|^2 d\xi)^{\frac{1}{2}}
 (\int |\hat{g}_\varepsilon(\xi)|^2 d\xi)^{\frac{1}{2}} \,
 \end{equation*}
 As the net is moderate there exists $m>0$ such that 
 $$( \int e^{2k\omega(|\xi|)} |\hat{g}_\varepsilon(\xi)|^2d\xi )^{\frac{1}{2}} \,=\,
 o(e^{m\omega(\frac{1}{\varepsilon})}) .$$
 Thus for any chosen $k>0$ and any chosen $l>0$ we can conclude that :
 
 $$ \int e^{k\omega(|\xi|)}|\hat{g}_\varepsilon(\xi)|^2 d\xi
 \,=\,o(e^{(-l+m)\omega(\frac{1}{\varepsilon})}) .$$
 As this holds for any choice of $k$ and as we can choose $l$ as large as we need we can now easily conclude.
 
  Roumieu case:
  
  In this case the hypothesis implies that: 
  \begin{equation*}
 \exists  l>0\,,\exists\,C_l>0\,,\mbox{s.t.} \,
(\int |\hat{g}_{\varepsilon}(\xi)|^2 d\xi)^{\frac{1}{2}} \,
  <\,C_le^{-l\omega(\frac{1}{\varepsilon})}.
  \end{equation*}

   The moderateness here can be translated into:
   \begin{equation*} 
  \forall \,m>0 \exists \,k>0 \mbox{s.t}.
 (\int (e^{2k\omega(|\xi|)}|\hat{g}_\varepsilon(\xi)|^2dx)^{\frac{1}{2}} \,=\,
 o(e^{m\omega(\frac{1}{\varepsilon})}) .
 \end{equation*}

 Again with H\"{o}lder's inequality and choosing $m <l$, we can conclude in the same way as before that:
 
 $$ \exists k>0\,,\exists\, \nu = l-m>0$$
 such that:
 $$ ,,\int e^{2k\omega(|\xi|)}|\hat{g}_\varepsilon(\xi)|^2d\xi
 \,=\,o(e^{-\nu\omega(\frac{1}{\varepsilon})})$$
 which in Roumieux case implies that our net is negligible (in Roumieu sense).
 \end{proof}
  
 \subsection{functorial aspects}
 
 As in all Colombeau type "generalizations" using a given 
 asymptotic scale, the way to define which mappings can be
 extended depends on the set of "moderate" fuctions for the given scale 
 in both our projective constructions, using asymptotic scales in 
  Beurling and Roumieu cases,we will prove that:
  \begin{proposition}
 for both asymptotic scales, the one used in Beurling case and the one used in the projective presentation of Roumieu case a moderate function cannot be of more than polynomial growth.
  \end{proposition}

  \begin{proof} 
a)Beurling case:
\\

Let $F$ be moderate for our asymptotic scale 
there would exis $\,k>0$
such that for large t,
 $$\,F(e^{\omega(t)})\,<\,e^{k\omega(t)} = (e^{\omega(t)})^k\,$$ Hence F is of at most polynomial growth

b)Roumieu case 
If $\,F\,$ had more than polynomial growth there would exist a sequence $\,(t_n)\,$, increasing to infinity 
such that 
$\,F(t_n)\,>\,t_n^n\,$
Consider $\,T_n\,$ such that
 $e^{\omega(T_n)}\,=\,t_n^n\,$ i.e. such that 
 $e^{\frac{1}{n}\omega(T_n)} \,= t_n$
 as in the construction of the sequence, we can impose inductively to the $t_n^n$ conditions such that $T_n$ satisfy the usual conditions for the construction of a weight function $\,\omega"<<\omega\,.$
 We can thus obtain a weight function $\omega"$ such that 
 $F(e^{\omega"(T_n)})\,>\,e^{\omega(T_n)}$ for large $n$.
 
 As the sequence $\,(T_n)\,$
 increases to infinity this contradicts the moderateness of $F$ because the image of the inverse of an element of a scale by a function which is moderate for the scale has to be bounded for large $t$ by an inverse of an element of our scale(in the projective presentation).  
 \end{proof}
 
\section{Strong and strict associations and regularities} 
 All those "Colombeau type algebras" have been introduced in order to embed various duals of spaces of various type regularities , distributions, ultradistributions etc in algebras in order to be able to face non linear type equations 
  This is usualy obtained by convolution of elements with compact support with nets of molifiers of some requested regularity,( see \cite{D.V.V.}).
   Thus those embeddings have the form of 
  $$ U(T)\,=\,[(T\star\phi_\varepsilon)]$$ where the net of molifiers:
 $$ ( \phi_\varepsilon(x)\,$$ is defined by 
 $$ \phi_\varepsilon(x)=\,
 \frac{1}{\varepsilon^n}\phi(\frac{x}{\varepsilon})$$
 Where $\phi$ has some adequate regularity.and
 $\hat{\phi}$ is $1$ on a neighbourhood of zero 
In the thesis of A.Debrowere and in \cite{D.V.V.} was proved that $T$ coresponds to an ultradifferentiable fuction if and only if its image belongs to a class of "regular " generalized ultradiferentiable functions
 Here we will prove that there are weaker sufficient conditions to that namely "strong associaton" for the Beurling case and "strict association" for the Roumieu case.
  For the case of ultradiferentiable functions where regularity was defined with the help of $\,(M_p)$ sequences satisfying some conditions this was done in \cite{P.R.S.Z}. Here we extend those ideas and methods in the case of Bj\"{o}rk type definitions.
 \subsection{simple association and strong and strict associations}
 \begin{definition}
 a weak form of association is to say that 
 two generalized functions , or a generalized function and an (ultra) distribution are "associated " if the différence of representative integrated with any test function pf the adquate type converges to zero (the same definition holds with a generalized function and a distribution or ultradistribution).
 \end{definition}

  Such a general definition does not impose any condition
  on the rate of convergence and does not allow to draw any conclusion of the regularity of a distribution which is simply associated to a "regular" generalized function thus we need somme more demanding definitions of  associations:
  
 \begin{definition}       
A  Beurling $(\omega)$-ultradistribution $\,T\,$ is said to be strongly associated to 
$\,[(g_\varepsilon)]\,\in\,\mathcal{G}^{(\omega)}\,$
If 
$$\,\forall\,K\subset\subset \Omega\,,\exists b>0\,
\mbox{s.t.} \,\forall \,\phi\,\in
 \mathcal{D}_K^{(\omega)}(\Omega)\,,
 <g_\varepsilon -T,\phi>\,=\,o(e^{-b\omega(\frac{1}{\varepsilon)}})\,$$
 In Roumieu case the last condition is replaced by the existence of a weight function $\,\omega_b\,<<\omega\,$
  such that: 
   $$ <g_\varepsilon -T,\phi>\,=\,o(e^{-\omega_b(\frac{1}{\varepsilon})})\,$$
   \end{definition}
  
  In order to compare regularities in Roumieu case we will need a stronger kind of association
  \begin{definition}
In both Beurling and Roumieux cases the strict association
is defined by the condition 
$$ [(<g_\varepsilon-T,\phi>)]\,=\,0$$ for any test function $\phi$.
i.e it is zero as an element of the ring of generalized constants of the corresponding type.
  In Roumieu case and the non projective presentation this is equivalent to the existence, for any test fuction $\phi$ of $l>0$ such that 
$$\,(<g_\varepsilon-T,\phi>)\,
=\,o(e^{-l\omega(\frac{1}{\varepsilon})})\,$$
In the projective presentation this is equivalent to the fact that: 
$$\,\forall\,\omega'<<\omega\,,(<g_\varepsilon-T,\phi>)\,
=\,o(e^{-\omega'(\frac{1}{\varepsilon})})\,$$

\end{definition}

Let us now remind the notions of regularities defined in the thesis of A.Debrouwere and in \cite{D.V.V.}
 for Beurling and Roumieu case:

\begin{definition}
BEURLING CASE REGULARITY.

An element $\,g\,=\,[(g_\varepsilon))]\,$
 is said to be $\{\omega\}$-regular 
 $(\,g\,\in\,\mathcal{G}^{(\omega),\infty}(\Omega)\,)$
 if a representing net $\,(g_\varepsilon,)$ belongs 
  to 
 $$\mathcal{E}^{(\omega),\infty}(\Omega) =
  \{(g_{\varepsilon}) \in \mathcal{E}^{(\omega)}_M (\Omega)  \, \mbox{s.t.} \forall K \subset \subset \Omega \,,
  \exists k>0 , \mbox{s.t.}
\forall h>0,||g_{\varepsilon}||_ {K,h} = o(e^{k\omega(\frac{1}{\varepsilon})} ) \} $$ 
\end{definition}.

In the definition of regularity for Roumieu case
for the non projective presentation of regularity we have also a shift of quantificators
 
\begin{definition}

ROUMIEU CASE REGULARITY (see \cite{D.V.V.}).
An element $ g=[(g_\varepsilon)]$ is said to be
 $\{\omega\}$-regular
 $(g\in\mathcal{G}^{\{\omega\},\infty}(\Omega)\,)$ if its representing net belongs to

$$\mathcal{E}^{\{\omega\},\infty}(\Omega) =
  \{(g_{\varepsilon}) \in \mathcal{E}^{\{\omega\}}_M (\Omega)  \, \mbox{s.t.} \forall K \subset \subset \Omega
  \exists h>0 , \mbox{s.t.}\,
\forall k>0\, ,||g_{\varepsilon}||_ {K,h} = o(e^{k\omega(\frac{1}{\varepsilon})} ) \} $$ 
\end{definition}

In the projective presentation we have another definition and we will have to prove that the two definitions are equivalent  equivalent.

    \begin{definition} 
An element $ g=[(g_\varepsilon)]$ is said to be
 $\{\omega\}$-regular
 $g\in\mathcal{G}^{\{\omega\},\infty}(\Omega)$ (if its representing net belongs to
 
 $$\tilde{\mathcal{E}}^{\{\omega\},\infty}(\Omega) =
  \{(g_{\varepsilon}) \in \mathcal{E}^{\{\omega\}}_M (\Omega)  \, \mbox{s.t.} \forall K \subset \subset \Omega\,,
  \exists \omega_1 << \omega\, , \mbox{s.t.} \,,
\forall \omega_2 <<\omega\,
||g_{\varepsilon}||_ {K,1}^{\omega_2} = o(e^{\omega_1(\frac{1}{\varepsilon})} ) \} $$ 
\end{definition}

   \begin{proposition}
The two definitions of Roumieu regularity are equivalent and 
$ \,\mathcal{E}^{\{\omega\},\infty}(\Omega),\ 
=\,\tilde{\mathcal{E}}^{\{\omega\},\infty}(\Omega)$
 \end{proposition}
  
\begin{proof}
As the properties we define here are local, we can without loss of generality suppose that there exists a compact subset $K$ of 
$\,\Omega\,$ which contains all the supports of $T$ and of all $\,g_\varepsilon$ considered here.
 let us first suppose that
  $(g_\varepsilon)\in \tilde{\mathcal{E}}^{\{\omega\},\infty}(\Omega)$
  If it did not belong also to 
  $\,\mathcal{E}^{\{\omega\},\infty}(\Omega)\,$ for  any 
  $\,n\in\,\NN$ there would exists $\,l_n>0\,$ such that 
$\,\int|\hat{g}_\varepsilon(\xi)|e^{\frac{1}{n}\omega(|\xi|)}d\xi\,$ is not
 $\,o(e^{l_n\omega(\frac{1}{\varepsilon})})\,$. 
 As for any given $\,\omega"<<\omega\,$
 and for $\,\varepsilon\,$ small enough 
 $\,e^{\omega"(\frac{1}{\varepsilon})}\,<\,e^{l_n\omega(\frac{1}{\varepsilon})} \,$ we can find 
 $\,\varepsilon_n$ as small as we want such that:
 
$$ \int|\hat{g_{\varepsilon_n}}(\xi)|e^{\frac{1}{n}\omega(|\xi|)}d\xi \,>\,
2ne^{l_n\omega(\frac{1}{\varepsilon_n})}\,>\,2ne^{\omega"(\frac{1}{\varepsilon_n})}\,. $$

Thus we can construct inductively a decreasing to zero sequence $\,(\varepsilon_n)\,$, and an increasing to infinity sequence $\,(T_n)\,$ such that:
   
  $\,\int_{|\xi|<T_n}\hat{g_{\varepsilon_n}}(\xi)|e^{\frac{1}{n}\omega(|\xi|)}d\xi\, 
  > ne^{\omega"(\frac{1}{\varepsilon_n})}.$
  
  In the inductive construction of the sequence $(T_n)$ and the sequence $\,(\varepsilon_n)\,$ 
  we can impose the condition of the construction of a weight function $\,\omega_1<<\omega\,$ and obtain 
   Thus
   
$$\int|\hat{g_{\varepsilon_n}}(\xi)|e^{\omega_1(|\xi|)}d\xi\,>\, ne^{\omega"(\frac{1}{\varepsilon_n})} $$

which implies that for any $\omega"<<\omega$ there exists 
$\,\omega_1\,$ such that:
 
$||g_{\varepsilon}||_ {K,1}^{\omega_1}\,$ is not $\,o(e^{\omega"(\frac{1}{\varepsilon})})\,$
which contradicts the hypothesis.
Conversely let us now suppose that 
$,(g_\varepsilon) \,
\in \mathcal{E}^{\{\omega\},\infty}(\Omega).$ by, puting 
$$\,t,\,=\; \frac{1}{\varepsilon}$$
This implies that there exists $\,l>0\,$ such that for all $k>0$

$ h(t) = ||g_{\frac{1}{t}}||_{K,l}^\omega $ is
 $\,o(e^{k\omega(t)})\,$ for all $\,k\,>\,0\,$
 
 By the fondamental lemma this implies
  that there exists $\,\omega"<<\omega\,$ such that
  $\,h(t) \,=\,o(e^{\omega"(t)} )  .$ 
  As for any 
  $\omega'<<\omega$ for $\,\varepsilon\,$ small enough ( t large enough)
  $||g_\varepsilon||_{K,1}^{\omega'} \,
  \leq\,||g_\varepsilon||_{K,l}^{\omega}$
  We can  now conclude that 
  there exist $\omega"<<\omega$ such that for any
   $\omega'<<\omega\,$
   we have :
   $$\,|g_{\varepsilon}||_ {K,1}^{\omega'}\,=
   o(e^{\omega"(\frac{1}{\varepsilon})})$$
   i.e. $\,g\,\in\,\tilde{\mathcal{E}}^{\{\omega\},\infty}(\Omega)\,$.
   \end{proof}
   
  \section{Comparison of regularities}

  In this section we will show how we can compare the regularity of an ulradistribution to the regularity of a generalized ultradifferentiable function under specific conditions . 
   We will pass into the "weight function" model some results already treated for the "$(M_p)$ sequence model" in \cite{P.R.S.Z}.
 Analogous results in different situations for also other kinds of regularities have been treated in  \cite{P.S.V.1}, in  \cite{P.S.V.2} and \cite{P.S.V.3}

  \subsection{Beurling case} 
   \begin{theorem}
 Let $\,T \, \\in \mathcal{D}'^{(\omega)}\, $  be a Beurling ultradistribution, strongly associated to a Beurling regular generalized ultradifferentiable function 
 $\,g=[(g_\varepsilon)]\,\in\,
 \mathcal{G}^{(\omega),\infty}(\Omega)\,$
 Then T is a regular Beurling  ultradifferentiable function
 $\,T \in \mathcal{E}^{(\omega)}(\Omega).\,$  
   
   \end{theorem}    
  
  \begin{proof}
 The regularity properties being local, in our case, we can suppose without loss of generality that there is a compact subset $\,K\subset\subset\,\Omega\,$ whose interior is  containing all the supports of $T$ and $(g_\varepsilon)$.Thus we can consider the Fourier trasforms and the translation of regularities and strong association properties  in terms of the Fourier transforms 
 of the elements we consider.
  The regularity of the net $\,(g_\varepsilon)\,$
  is expressed in :
  $$\,\exists k>0\, \mbox{s.t.}\, \forall\,l>o, \exists C_l>0\,\mbox{s.t.}
sup(|\hat{g}_\varepsilon(\xi)|e^{l\omega(|\xi|})\,<\,
C_le^{k\omega(\frac{1}{\varepsilon)})}\,.$$
 This amounts to
 $$\forall\, \xi\,\in\,\RR^n \|\hat{g}_\varepsilon(\xi)|\,\leq\,
 C_le^{k\omega(\frac{1}{\varepsilon)})-l\omega(|\xi|}
 $$ 

 The regularity of the ultradistribution $\,T\,$ is expressed by 
 $$\,\forall \lambda>0\,
 ,|\hat{T}(\xi)|e^{\lambda\omega,|\xi|)}\,
 <\,\infty\\,,$$ 
 Let $\psi$ be a test function equal to one on the compact $K.$  Considering the norms of 
 $ \psi(x)e^{ix\xi}$' we can easily prove in an analogous way as in \cite{D.V.V.} that                                               
 in terms of Fourier transforms, the Beurling strong association is expressed  by:
 
 $$\exists b>0\,\mbox{s.t.} \exists\, h>0\,\exists\,C_h>0
 \,\mbox{s.t}\,
 |\hat{T}(\xi)-\hat{g}_\varepsilon(\xi)|\,\leq\,
 C_he^{h\omega(|\xi|)-b\omega(\frac{1}{\varepsilon})}.$$
 
 Thus
 \begin{equation*}
  e^{\lambda\omega(|\xi|)}|\hat{T}(\xi)|\,\leq\,
 e^{\lambda\omega,|\xi|)}|(\hat{T}(\xi)-
 \hat{g_\varepsilon} (\xi))|\,+
 \,|\hat{g_\varepsilon}(\xi)|e^{l\omega(|\xi|)})\,\leq
 \,C_he^{\lambda\omega(|\xi|)+h\omega(|\xi|)
 -b\omega(\frac{1}{\varepsilon})}\,+\,
 C_le^{(\lambda-l)\omega(|\xi|)+
 k\omega(\frac{1}{\varepsilon})}\,.
 \end{equation*}
 
 This holds for any couple $\,(\xi,\varepsilon)\,.$
 Thus for any given $\xi$, choose $\,\varepsilon\,$ to be such that
  $\lambda\omega(|\xi|)+h\omega(|\xi|)-
  b\omega(\frac{1}{\varepsilon})\,=\,0\,$ (this is posssible for $\xi$ large enough.)
  Thus the first exponant is zero .Putting now
  $\omega(\frac{1}{\varepsilon})\,
  =\,\frac{\lambda+h}{b}\omega(\xi)\,$ in the second exponent we obtain $(\lambda-l+k\frac{\lambda+h}{b})\omega(|\xi|)$
  
  Thus if we choose
   $\,l\,=\,\lambda+k\frac{\lambda+h}{b} \,$
    the second exponent is also zero.
     We have thus verified  on the Fourier tranform of $T$ the condition necessary and sufficient to conclude that it is a regular ulradifferentiable function.
     \end{proof}
     \subsection{Roumieu case}

 In Roumieu case an analogous theorem does not hold because:
 \begin{proposition} 
There exists a Roumieu regular generalized function  
$$g=[(g_\varepsilon)] \,\in\,
\mathcal{G}^{\{\omega\},\infty}$$ strongly associated to an irregular Roumieu ultradistribution. 
 \end{proposition}   
    
 \begin{proof}   
 In order to show this,  We will construct a regular net
    $ g=[(g_\varepsilon)]$ strongly associated to the dirac distribution $\delta$.
    Let $\phi_\varepsilon$ be a net of molifiers with the help of which is obtained the embedding of 
    $$ \mathcal{D}^{' \{\omega\}} (\RR^n)$$ into
    $$ \mathcal{G}^{\{\omega\}} (\RR^n)$$ (see\cite{D.V.V.} )
    Let us choose arbitrarily $\omega_1 <<\omega$
    and consider the "slowing down " mapping $\eta$ defined by :
    $$\omega_1(\frac{1}{\varepsilon}) =
    \omega(\frac{1}{\eta(\varepsilon)}) \,,$$ and let us now consider the net 
    $$ (g_\varepsilon) = (\phi_{\eta(\varepsilon)})$$
    As the net $\phi_\varepsilon$ is strictly associated to $\delta$ one can immediately see that the net 
    $$ (g_\varepsilon) = \phi_{\eta(\varepsilon)}$$
    is strongly associated to $\delta$.
    Notice now that as the net $\phi_\varepsilon$
    is moderate for any compact $K$ containing zero in its interior:
    $$ \forall \omega' <<\omega,\,
    \exists \,\omega''<<\omega \,\mbox{s.t.}\,
    ||\phi_\varepsilon||_{K,\omega'}\,=\,
    o(e^{\omega''(\frac{1}{\varepsilon} )})$$
    But this implies that it is also 
    $$o(e^{\omega(\frac{1}{\varepsilon} )})$$
    which implies that 
    $$\forall \omega' <<\omega,\,
    ||g_\varepsilon||_{K,\omega'}\,=\,
    ||\phi_{\eta(\varepsilon)}||_{K,\omega'}\,=\,
    o(e^{\omega(\frac{1}{\eta(\varepsilon)})}) \,
    =\,o(e^{\omega_1(\frac{1}{\varepsilon})} )$$
    which establishes the regularity of $g$
    \end{proof}
   Thus we need the strict association :

     Let us now pass to the Roumieu case with strict association:
  
 \begin{theorem}
 Let $T$ be a Roumieu ultradistribution
  $(T\in\mathcal{D'}^{\{\omega\}}(\Omega))\,$.If it is 
 strictly associated to a regular Roumieu genereralized ultradifferentiable function $g=[(g_\varepsilon)]$ ,i.e.
 $(g\in\mathcal{G}^{\{\omega\},\infty}(\Omega))$
  then it is a Roumieu regular ultradifferentiale function.
  $( T\in \mathcal{E}^{\{\omega\}}(\Omega))\,$
  
\end{theorem}
   \begin{proof} 
 As before we can suppose without loss of generaliry that 
 all elements we consider have compact support included in some compact subset $\,K'\,$ included int the interior of some compact subset $K$ of $\omega$
 As the family of elements of $\mathcal{D}^{\{\omega\}} (\Omega)$ with support inside $K$ is the union of all
 $\mathcal{D}^{(\omega),l}(K)$

 $$ \,\forall\,l>0 \,,\forall\,\rho \,\in\,\mathcal{D}^{(\omega),l}(K)  
 [<T-g_\varepsilon,\rho\>] \,=\,0 $$
  Taking onto consideration that for all our definitions we can choose any of the families of norms (see chapter 2)                   
 
The Roumieu regularity of the Roumieu  generalized function with compact support
$ g\,=\,[(g_\varepsilon)]\,$ is translated in Fourier trasform setting  by 
$$\,\exists\,l>0\,\mbox{s.t}\,\forall\,k>0\,\exists
\,C_k >0\,\mbox{s.t.}\,
|\hat{g}_\varepsilon(\xi)|\,<\,
C_ke^{k\omega(\frac{1}{\varepsilon)}-l\omega(|\xi|)}\,$$
 
 to translate the strict asociation into Fourier setting we first have to prove the following uniformity lemma:
 
 \begin{lemma}
 If $(S_{\varepsilon})$ is a net of Roumieu ultradistributions with support in the interior of some fixed compact subset $K'\subset\subset K\subset\subset\,\Omega$ such that for any test fuction 
 $$\rho \,\in\,\mathcal{D}^{\{\omega\}} \,$$we have the fallowing property

 $$[(<S_\varepsilon,\rho>)] \,=\,0\,$$
then 
\begin{equation*}
\forall\,h>0\,,\exists b>0 \,,\exists, C>0 \,,\mbox{s.t.}\,
\exists\,m\,\in \,\NN \,,\mbox{s.t.}\,
\forall\,\rho\,\in\,\mathcal{D}^{(\omega),h}_K\,
\forall \varepsilon \,\leq \,\frac{1}{m}
<S_\varepsilon,\rho>\,\leq\,
Ce^{-b\omega(\frac{1}{\varepsilon})}||\rho||_{K,h} \,.
\end{equation*}
\end{lemma}

\begin{proof}

Let us consider in the Banach 
$$\,\mathcal{D}'^{(\omega),h)}_K
 $$.
  The countable family of closed subsets
  $\{ F^{P,q,m}\} ,(P,q,m)\,\in\NN\times\NN \times\NN\, $
  defined by
  $$ F^{P,q,m}\,=\,\{\rho\,\in\,
  \mbox{s.t.}\,
  \forall \varepsilon \,\leq \,\frac{1}{m}\,,
 | <S_\varepsilon,\rho> |\,\leq\,
Pe^{-\frac{1}{q}\omega(\frac{1}{\varepsilon})}\}\,,$$
 The countable union of those closed symetric sets is the whole Banach space $\mathcal{D}^{(\omega),h}_K$. Thus by the Baire theorem there is at least one of them with no void interior. Thus there exists 
 $(P,q,m)\,\in\NN\times\NN\times \NN, $ and a ball $B(\rho_1,r)$
  in $\mathcal{D}^{(\omega),h}_K$ such that 
  $$  B(\rho_1,r)\,\subset\,F^{P,q}$$
  and clearly the same holds for$B(-\rho_1,r).$
   hence 
   $$ B(0,r)\,\subset B(\rho_1,r)+B(-\rho_1,r)\,
   \subset F^{2p,q,m}$$ and now we can easily conclude with $ m=m$, $b=\frac{1}{q}$ and an adequate C>0\,,
   \end{proof}

   Now taking for $(S_\varepsilon)$ the net
   $( T-g_\varepsilon )$,and ,as before,   
   $\rho = \psi(x)e^{ix\xi}$where $psi$ is a test function equal to one on $K$ ,  we can easiy verify that the strict associaton is translated in Fourier setting by :
 
$$\,\forall h>0\,,\exists C>0 ,\exists b>0 , \exists m>0\,,\exists 
\mbox{s.t.}\,\forall \varepsilon \leq \frac{1}{m}\,,
|\hat{T}(\xi)-\hat{g}_\varepsilon(\xi)|\,\leq\,
 Ce^{h\omega(|\xi|)-b\omega(\frac{1}{\varepsilon})}.$$
 
 The Roumieu regularity of $g$ is translated in Fourier setting by 
$$\exists l>0,\mbox{s.t.}  \forall k>0 \,,\exists m_k >0\, \exists C_k>0\,,
\mbox{s.t.}\,\forall \,\varepsilon \leq \frac{1}{m_k}\,,
 |\hat{g_\varepsilon}(\xi)|\,\leq\,
C_ke^{-l\omega(|\xi|) + k\omega(\frac{1}{\varepsilon})}
\,$$         
 
 As before
 $$|\hat{T}(\xi)|\,\leq\,
 |(\hat{T}(\xi)- \hat{g}_\varepsilon (\xi))|\,+
 \,|\hat{g}_\varepsilon(\xi)|\,\leq
 \,Ce^{+h\omega(|\xi|)
 -b\omega(\frac{1}{\varepsilon})}\,+\,
 C_le^{-l\omega(|\xi|)+
 k\omega(\frac{1}{\varepsilon})}\,.$$ 
 Let us first for given $\,\xi, $ choose $\,\varepsilon\,$ in a way that 
 $b\omega(\frac{1}{\varepsilon})\,=\,2h\omega(|\xi|)\,
 $ 
 
 and thus the first exponent becomes $-h\omega(|\xi|)$                
 
 Putting this value of $\omega(\frac{1}{\varepsilon})$
 in the second exponent we obtain 
 $\,(-l+\frac{2kh}{b})\omega(|\xi|)$
 Choose now k to be equal to $ \frac{bl}{4h}$.We obtain an exponent:
 $-\frac{l}{2}\omega(|\xi|)$
 Puting now 
 $\lambda\,=\,min(h,\frac{l}{2})$
 We obtain the existence of a constant $C_{\lambda}$ such that 
 
 $\,|\hat{T}(\xi)|e^{\lambda\omega(|\xi|)}\,<\,C_{\lambda}\,$ 
  which proves the regularity of $T$.
 
 \end{proof}

      However strong association in Roumieu case can in some cases give information for lower regularities 
       if for example the $\,\omega_b\,$ used for strong association belongs to the same class as the $\Omega"$ used for the regularity of $\, g = [(g_\varepsilon)]$
       we can conclude to a $\{\omega"\}$ regularity for $T$.
       There is also another case where, with a stronger hypothesis on strong association, we can conclude:
       \begin{theorem}
Let $T$ and $g=[(g_\varepsilon)]$ be respectively a Roumieu ultradistribution and a regular Roumieu generalized ultradifferentiable function satisfying the following properties: 
A) they both have compact supports inside the interior of a fixed compact subset $K$.

B) $\omega_\infty$ is such that 
$$\forall \omega' << \omega \,,
||g_\varepsilon||_{K,\omega'}\,=\,
o(e^{\omega_\infty (\frac{1}{\varepsilon})}) .$$
C) for $\varepsilon$  small enough the following holds:

\begin{equation*}
\
\exists \omega_b >> \omega_\infty \mbox{s.t.} \,
\forall\,\omega_2 <<\omega \,\exists C>0 \mbox{s.t.}
\forall \,\rho\,\in \mathcal{D}^{\omega_2}(\Omega )\,
 \,|<T-g_\varepsilon,\rho>|\,\leq\,
Ce^{-\omega_b(\frac{1}{\varepsilon})} 
||\rho||_{K,\omega_2}
\end{equation*}
(this property is called "$\omega_b-$ R-strong association) 
Then T is Roumieu regular 
\end{theorem}

\begin{proof}
As we have compact supports all inside $K$ we can pass to Fourier transform setting . In the sequel of our proofs remember  we can use freely any choice of the equivalent family of seminorms for the norms on test functions. 
 By Wiener type theorems 
  we only have to prove that 
  $$\forall \,\omega_\lambda<<\omega\,,
  sup(e^{\omega_\lambda (|\xi|)|}|\hat{T}(\xi)|) < \infty
  $$
  In Fourier setting the hypothesis B)
  amounts to 
  B')$$ \forall \omega_h<<\omega\,\exists C_h \ ,
  \mbox{s.t.}\, \forall \xi\in \RR^n
  | \hat{g_\varepsilon}(\xi)|\,\leq \, 
  C_h e^{\omega_\infty(\frac{1}{\varepsilon})-
  \omega_h(|\xi|)}.
  $$

  The hypothesis C) amounts to 
  C')$$ \forall \omega_l << \omega \,,\exists C_l>0\,
  \mbox{s.t.}\,
  |\hat{T}(\xi)-\hat{g_\varepsilon}(xi)|\,\leq \,
\,.$$  Thus in an analogous way as in  the proof of Beurling case
  \begin{equation*}
  |\hat{T}(\xi)|\,\leq\,
  |\hat{T}(\xi)-\hat{g_\varepsilon}(\xi)|\,+\,
  |\hat{g_\varepsilon}(\xi)|\,\leq \,
  C_le^{\omega_l([\xi|]-\omega_b(\frac{1}{\varepsilon})}
  \,+\,C_h e^{\omega_\infty(\frac{1}{\varepsilon})-\omega_h(|\xi|)}
  \end{equation*}
  
  Thus for any $\omega_\lambda$ we have :
  \begin{equation*}
 e^{\omega_\lambda (|\xi|)} |\hat{T}(\xi)| \,\leq
  C_h e^{\omega_\lambda(|\xi|)+
  \omega_\infty(\frac{1}{\varepsilon})-
  \omega_h(|\xi|)} \,+\,
   c_le^{\omega_\lambda(|\xi|)  +   \omega_l(|\xi|)-\omega_b(\frac{1}{\varepsilon})}
   \end{equation*}
   Put $\omega'_\lambda = \omega_\lambda +\omega_l$
   For $|\xi|$ large enough consider $\varepsilon$ such that $$\omega_b(\frac{1}{\varepsilon})\,=\,\omega'_\lambda(|\xi|)$$ i.e.
   $$\frac{1}{\varepsilon}\,=\,
   (\omega_b)^{-1}(\omega'_\lambda(|\xi|).$$
   
   This implies that the second term of the above equality remains bounded 
   as for the first exponent , putting this value of $\varepsilon $ we find 
   an exponent:
   
$$  \omega_\lambda(|\xi|)+
  \omega_\infty(\omega_b)^{-1}(\omega'_\lambda(|\xi|))-
  \omega_h(|\xi|) \,$$
  As $$\omega_\infty << \omega_b$$ we can verify
  that $$\omega_\infty(\omega_b)^{-1}(\omega'_\lambda) <<
  \omega $$
   Thus we can choose $\omega_h$ such that 
    $$\omega_\lambda(|\xi|)+   \omega_\infty(\omega_b)^{-1}(\omega'_\lambda )\,<<\,\omega_h$$ and conclude that also the first term is bounded 
    which proves the regularity of $T$ 
    \end{proof}     

\section{equalities}

In this section we investigate in our setting, some questions which were investigated in the case of usual Colombeau generalized fuctions in 
 \cite{P.S.V.} ,and in the ultradifferentiable with sequences model in  \cite{P.R.S.Z}..We investigate  some  cases in which we can conclude that two elements are equal or that en element is equivalent  to a generalized constant .

\begin{theorem}
Let f be an element of
 $\, \mathcal{G}^{(\omega)}(\RR^n)\,$
(resp of $\,\mathcal{G}^{\{\omega\}}(\RR^n)\,)$ 
Then if 
$$\forall x\in\,\RR^n\,f(.+x) -f(.) \,
=\,0 
$$ ,(i.e. f is invariant by translations) , then $f$ is equivalent to a generalized Beurling constant (resp generalized Roumieu constant.)
 \end{theorem}

\begin{proof}
 In both cases the strategy of the proof is the following:
 We express the hypothesis of negligibility in terms of the neglibilty of the nets of $\,L_2\,$ norms on the balls $\,B(0,2R)\,$ ($R$ going to infinity ( theorem 5.1.1,) and with Baire theorem  ,we find a "uniformity" of this for the bounds of the  element $\,h\,$ of the translation..  We use the fact that the definition of negligibility holds for all derivatives up to any givern integer and the fact that bounds of $ H_s   $ norms for
  $\,s\,>\, \frac{n}{2}$ on compact subsets
 allow to obtain bounds of sup-norms in the ball for
  $\,|f_\varepsilon(t)\,-\,f_\varepsilon(0)|\,.$ This now allows to come back to $L_2$ norms and thus obtain the neglibility 
   of the net  $\,|f_\varepsilon(t)\,-\,f_\varepsilon(0)|\,$in the sense of our definitions. 
   \\
1) BEURLING CASE

In Beurling case the hypothesis amounts to saying that for any
 $\, x \,\in\,\RR^n$  the net
$$\,f_\varepsilon(t+x)\,-\,f_\varepsilon(t)$$ isa negligible net of ulttradifferentiable functions .
 This, in our setting and in view of the charachterisation of negligibility in theorem 5.11,and the fact that the same negligibility conditions hold for all derivatives  amounts to

$$\,\forall \,B(0,R\,,\forall\,p\,\in \,\NN\,,\forall x \,\in\,B(0,R)\,,\exists \,m\,\in \NN \,\mbox{s.t.}\,
\,\forall \,\varepsilon \,<\,\frac{1}{m}\,,
e^{p\omega(\frac{1}{\varepsilon})}||f_\varepsilon(t+x)\,-\,
f_\varepsilon(t) ||_{H_s,2R}\,\leq\,1 $$

 where $\,||.||_{H_s,2R}\,$ is the local $\,H_s\,$ norm in the ball  $B(0,2R)\,.$
 We will now choose $\, s>\frac{n}{2}\,$ and we will prove that :
$$,\forall \,\varepsilon \,<\,\frac{1}{m}\,, \forall x\,\in\,B(0,R)\,,
e^{p\omega(\frac{1}{\varepsilon}}||f_\varepsilon(t+x)\,-\,
f_\varepsilon(t) ||_{H_s,2R}\,\leq\,C\,.$$

\  which will allow us to conclude to the negligibility of the net of local  $^L_2$ norms of
   $f_\varepsilon(t) \,-\,f\\varepsilon(0)$ In the balls  
   $\,B(0,R)\,$
    
    As this holds for all such balls ,  is the proof that our net is equivalent to the  generalized constant net
     $  \,f_\varepsilon(0)\,.$ 
  This is just the application of the negligibility ctiterion of theorem 4.11.
 
 For any given ball $B(0,R)$ , for any couple  of integers 
 $(p,l,)$, consider the set 
$$F_{p,l}= \{x\in B(0,2R)\,,\forall \varepsilon \leq \frac{1}{l}\,,
e^{p\omega(\frac{1}{\varepsilon})} ||(f_\varepsilon(x+t)-f_\varepsilon (t) ||_{H_s,2R})\,
\leq\,1\}$$\

For any fixed $p$, the union of all $F_{p,l}$  covers also the ball $B(0,R)$,  thus By Baire theorem, there exists $l$ such that the interior of $F_{p,l}$ is not void, thus contains some ball $B(x_1,r_1)$ with $r_1< R $] and
 $\,x_1\,\in \,B((0,R)\,$

By noticing now that for any
 $\,h\,\mbox{s.t.}|h|<r_1\,$ we have the following inequality :

$$|f_\varepsilon (t+h)-f_\varepsilon(t)|\,\leq\,
|f_\varepsilon(x_1+h+t)-f(h+t)|
\,+\,|f_\varepsilon (x_1+h+t)-|f(t)|\,.$$ We can now easily conclude that for any translation by $h$ of module less than
 $r_1$

$$e^{p\omega(\frac{1}{\varepsilon})}||(f_\varepsilon(h+t)-f_\varepsilon (t) ||_{Hs,2R})\,\leq\,2 \,.$$ as any transaltion as large as needeed is a composition of translations of $h$ s.t. $h<r_1$ 

we can conclude  that :
$$\,\forall \,B(0,R)\,\forall\,p\,\in \,\NN\,,\exists\, C>0 \,\exists \,m\,\in \NN \,\mbox{s.t.}\,
\,\forall \,\varepsilon \,<\,\frac{1}{m}\,\forall x\,\in\,B(0,R)\,,
e^{p\omega(\frac{1}{\varepsilon}}||f_\varepsilon(t+x)\,-\,
f_\varepsilon(t) ||_{H_s,2R}\,\leq\,C,.$$

  Thus, we can conclude that analogous bounds hold for the supremum norm  in any ball $\,B(0,2R)\,.$

Thus 
\begin{equation*}
\exists C>0,,\ \mbox{s.t.}\,\sup_{t\in B(0,R),h\in B(0,R)}
\{e^{p\omega(\frac{1}{\varepsilon})}
|f_\varepsilon(h+t)-f_\varepsilon (t) |\}
\,\leq\,C.
\end{equation*}

As This holds for all $p$
We can  conclude that the sup-norms in the balls
$\,B(0,R)\,$  are negligible and in this ball the sup-norm of

 $$\,(f_\varepsilon(t) - f_\varepsilon(0)) $$
 is a negligible net and thus the same holds for its  $L_2$ norm 
 in the ball which implies (as this holds in all such balls) that our generalized function is equal to the generalized constant 
 $\,[(f_\varepsilon(0)))]\,$ in the sense of our definitions .

ROUMIEU CASE

In Roumieu case  , with the same reasoning the negligibility holds also for all derivatives, the hypothesis amounts to:
$$\,\forall\, R>0 \,\forall x \in B(0,R)\,,\exists p\,\in \NN,  \,,\,m\,\in \NN\,,
\mbox{s.t}\,
\,,\forall \varepsilon\leq \frac{1}{m}\,,
(e^{\frac{1}{p}\omega(\frac{1}{\varepsilon})}
||(f_\varepsilon(x+t)-f_\varepsilon (t) ||_{H_s,2R})\leq\,1.$$

.
For any couple of integers $(p,l)$, consider the set closed  set 
$$H_{p,l}= \{x\in B(0,R)\,,\forall\varepsilon \leq \frac{1}{l}\,e^{\frac{1}{p}\omega(\frac{1}{\varepsilon})}
||(f_\varepsilon(x+t)-f_\varepsilon (t))||_{H_s,R}\,\leq\,1\}$$
The union of this countable family of closed  subsets  contains the ball$\,B(0,R)\,$
thus by ||Baire theorem there exists at least one of them of non void interior 
Thus there exist $(p,l) \in \NN^2$ containing some ball
$B(x_2, r_2)$ with $\,r_2<R\,.$  Notice that  as in Beurling case that  
$$|f_\varepsilon (t+h)-f_\varepsilon(t)|\,\leq\,
|f_\varepsilon(x_1+h+t)-f(h+t)|
\,+\,|f_\varepsilon (x_1+h+t)-|f(t)|\,.$$
Thus as before 
for $\,\varepsilon\leq \frac{1}{l}$ and $h, \mbox{s.t.} <|h|<R$
we have $$
e^{\frac{1}{p}\omega(\frac{1}{\varepsilon})}||(|f_\varepsilon (t+h)-f_\varepsilon(t)||_{H_s,R}\,\leq\,C  $$

 (any translation can be considered as the composition of translations of $h$ with $\,|h|<r_1,$

Now we conclude exactly as in the Beurling case 

 and by the same arguments  conclude thar the net is equivalent to the net $(f_\varepsilon(o)\,.$ in the Roumieu sense.
Thus It is a Roumieu  generalized constant.

\end{proof}

 One can easily see that if a net $(h_\varepsilon)$
  is such that for all test functions $\rho$
  we have $$[(<h_\varepsilon,\rho>)]\,=\,0$$ this does not imply that the net is negligible (for example take the difference $$ h_\varepsilon \,= \,
  \phi_{1, \varepsilon}\,- \,\phi_{2,\varepsilon}$$
   of two different nets used for the embedding of ultradistributions ; However if we impose this to hold 
   for all test functions  of some "lower" regularity we have the following results, analogous to what was done in 
   \cite{P.S.V.} for usual Colombeau generalized fubctions and in \cite{P.R.S.Z} for the sequence $M_p\,,$ model of ultradifferentiable functions (the results are analogous, but the proof methods are quite different).

 For Beurling case:

\begin{theorem}
     If an element ,$ \,h \,=\, [(h_\varepsilon)]\,$ of
 $\mathcal{G}^{(\omega)}(\Omega) \,$ is such that 
 $$\exists k>0\,,\,\mbox{s.t}\,,\forall \rho 
 \in \mathcal{D}^{(\omega),k}(\Omega)\,
 [(<h_\varepsilon,\rho)>)]\,=\,0\,$$
 then $h=0$
\end{theorem}

\begin{proof}
We will use the equivalence of families of seminorms already given in the first chapter and use here mainly the
$"L^2" $ seminorms.Clearly by partition of unity we can without loss of generality suppose that there is a compact subset $K$ of $\Omega$ such that all 
$h_\varepsilon$ have compact supports included in the interior of $\,K\,.$ For simplicity of natations the norms we use will be simply noted $\,|| .||_{K,m}\,$ and here it will mean the norms defined  by $L_2$  estimations .

As usually by partition of unity we can suppose, without loss of generality, that all functions considered here, have support in the interior of some ball $B$
 The hypothesis is equivalent to saying that

 One can easily verify that The hypothesis amounts to:

 \begin{equation*}
\forall \,a>0\,\,\forall\,\rho \, \in \mathcal{D}_B^{(\omega),k}\,,
\exists M\in\NN \, ,\exists l \,\in \NN\,\mbox{s.t.}\,
\forall \,\varepsilon \,\leq\,\frac{1}{l}\,,
 |<h_\varepsilon,\rho>|\,\leq\,
 Me^{-a\omega(\frac{1}{\varepsilon})}
 \end{equation*}
 
 Let us consider now the following  family of closed symetric subsets of the Banach space
  $\mathcal{D}_B^{(\omega),k}$,parametrized by integers $a,M,l$
 
$$F_{a,M,l} \,=\, \{\, \rho \, \in \mathcal{D}_{B^{(\omega),k}} \, \mbox{s.t.}\,\forall \,\varepsilon \,\leq\,\frac{1}{l}\, |<g_\varepsilon, \rho>| \,\leq\, M  e^{-a\omega(\frac{1}{\varepsilon})}\, \}.$$

 For any given $a>0$ this family is countable and covers the 
 Banach space $\mathcal{D}_B^{(\omega),k)}\,$.
 Thus by Baire theorem there exists at least one of them
 $F_{a,M_0,l_0}$ with non-empty interior and hence it containes a ball $B(\rho_0 , r)$ As it is symetric it contains also 

$B(-\rho_0 , r)$ . Thus 
$$ \,B(0,r) \,\subset \,B(\rho_0 , r)\,+\,B(-\rho_0,r)\,
\subset\,F_{a,2M,l}\,$$
Now we can easily conclude that

  $$\,\exists s>0\, \exists C_a>0\,\mbox{s.t.}\,\forall \rho \,\in \,
  \mathcal{D}_B^{(\omega),k}
  \forall\, \varepsilon \,\leq\,
  \frac{1}{s}\,
  |<h_\varepsilon,\rho>|\,\leq \,
  C_a e^{-a\omega(\frac{1}{\varepsilon})} ||\rho||_{B,k} \,.
  $$  
  Now for every $\,\varepsilon\,\leq \,\frac{1}{s}$ 
   put $\rho\,=\,h_\varepsilon$ . thus for $\varepsilon$ small enough:
   \begin{equation*}
  |<h_\varepsilon,h_\varepsilon>|\,\leq \,
  C_a e^{-a\omega(\frac{1}{\varepsilon})} ||h_\varepsilon||_{B,k} 
 \end{equation*}
 But as $h$ is moderate there exists $m$ (depending only  on $k$)
 such that  $$ ||h_\varepsilon||_{B,k}\,=\
 ,o(e^{m\omega(\frac{1}{\varepsilon})}) $$
 
 Thus

 $$|<h_\varepsilon,h_\varepsilon>|\,
 =\,o(e^{(-a+m)\omega(\frac{1}{\varepsilon})} \,.$$
 
 As this holds for any $a$ as large as necessary , clearly 
 $\,|<h_\varepsilon,h_\varepsilon>|\,$ is a negligible net and thus $\,[(h_\varepsilon)] \,=\,0\,$ 
 
 \end{proof}
 
 The Roumieu case is slightly different  
 
 \begin{theorem}
   If an element $h=[(h_\varepsilon)]$ of
 $\mathcal{G}^{\{\omega\}}(\Omega) \,$ is such that 
 $$\,\exists \,,\omega_1 <<\omega \,\mbox{s.t} \, \forall \rho 
 \in \mathcal{D}^{(\omega_1)}(\Omega)\,
 [(<h_\varepsilon,\rho)>)]\,=\,0  $$
 then $h=0$
 \end{theorem}
 
 \begin{proof}
 We can make, without loss of generality, the same assumptions for the supports as in the proof of the previous theorem. and consider only fuctions with support on some given ball $\,B\,$
 Our assumption implies that :
  \begin{equation*}
 \forall \rho
 \in \mathcal{D}_B^{(omega_1)}\,,\exists l>0 \,\mbox{s.t.}
 \,<h_\varepsilon,\rho>\,=\,
  o(e^{-l\omega(\frac{1}{\varepsilon})})
  \end{equation*}
  
  Consider now the following closed symetric subsets of 
  $\,\mathcal{D}_B^{(\omega_1)}\,$
  indexed by integers $l,P,m$
  $$ \,H_(l,P,m)\,=\, \{\,\rho\,\in\,\mathcal{D}_B^{(\omega_1)}\,
  \mbox{s.t.}\,\forall\, \varepsilon\,\leq \,
  \frac{1,}{\varepsilon}\,,  
  |<h_\varepsilon,\rho>|\,\leq\,
   Pe^{-\frac{1}{l}\omega(\frac{1}{\varepsilon})}$$

 the countable union of those closed symetric sets 
  is the complete metrisable space 
  $\,\mathcal{D}_B^{(\omega_1)}\,,$  thus there exists one of them $\, H_{P_1,l_1,m_1}\,$    with non void interior .As our sets are symetric it contains a neighbourhood
  $\,V\,$ of some $\,\rho_1\,$ and its symetric thus by addition we see that $\, H_{2P_1,l_1,m_1}\,$  contains a neighbourhood of zero.  By the usual steps now we conclude that there exists $p, l , C$ such that 
  for all $\,\rho\,\in\, D_K^{(\,\omega_1)}\,$, and 
  $\varepsilon\,\leq\,\frac{1}{m_1}\,$
  we have:
  $$\, |<h_\varepsilon,\rho>|\,\leq\,
 Ce^{-\frac{1}{l}\omega(\frac{1}{\epsilon})}||\rho||_{K,p\omega_1}\,$$
   Applying this for every $\,\varepsilon\,\leq\,\frac{1}{l}  $ 
   to
    $\,g_\varepsilon\,$ we prove as previously , taking into consideration the fact that $\,(h_\varepsilon) \, $ is a moderate net that there exists $\,\omega_2\,<<\omega\,$ such that
    $$|<,(h_\varepsilon,,(h_\varepsilon>|\,=\,
    o(e^{-\frac{1}{l}\omega(\frac{1}{\varepsilon})+
    \omega_2(\frac{1}{\varepsilon})})$$
    but as $$\omega_2 \,<<\,\omega $$
    this implies that if $\,0\,<\,a\,< \,\frac{1}{l}$ then 
    $$|<h_\varepsilon, h_\varepsilon>|\,=\,o(e^{-a\omega(\frac{1}{\varepsilon}})$$ which implies in that it is a negligible net in Roumieu senss and thus $\,[h_\varepsilon] \,=\,0\,.$

 \end{proof}

  However, as in usual Colombeau algebras \cite{P.S.V.} or in the ultradifferentiable constructions  defined by  sequences $(M_p)$ \cite{P.R.S.Z}), for the cases of regular elements we have the following propositions:
  In Beurling case:
  \begin{theorem}

  If $\,[(g_\varepsilon)]=g\,$ belongs to 
  $\,\mathcal{G}^{(\omega),\infty}(\Omega)\,,$ 
  and 
  $$\, \forall \, \rho \in \,
   \mathcal{D}^{(\omega)}(\Omega) \, ,
   (<g_\varepsilon,\rho>) \,\in \,
   \mathcal{N}^{(\omega)}(\Omega)\,$$
    then $\,g\,=\,0 .$
    \end{theorem}
  \begin{proof}

  As previously we can by partition of unity suppose without loss of generality that there exists a compact subset $\,K \subset\subset \Omega$ whose interior contains all the supports of $g_\varepsilon \,.$
  We have to prove that under those conditions that:
  \begin{equation*}
  \forall\, a>0\, |<g_\varepsilon, g_\varepsilon>|\,=\,
  |<\hat{g}_\varepsilon,\hat{g}_\varepsilon>|\,=
  \,o(e^{-a\omega(\frac{1}{\varepsilon})}.)
  \end{equation*}
  By the regularity hypothesis we know that there exists 
  $h>0$ such that for any $k>0$
  $$||g_\varepsilon||_{K,k} =
   o(e^{h\omega(\frac{1}{\varepsilon})})$$
   (Here we use the family of the $L^2$ norms for the topology of our ultradifferential functions).
 For any chosen $a$ the family of linear continuous forms, on the complete metrisable space topologized by a countable family of seminorms 
  $\mathcal{D}_K^{(\omega)} $ defined by 
 $ \,\rho \rightarrow
  <e^{a\omega(\frac{1}{\varepsilon})}  g_\varepsilon, \rho> $ converges to zero.  Thus by the same techniques as before  
  we can prove that there exist $l>0, C>0,m \in \NN,m>0$
  such that for $\,\varepsilon\,\leq\,\frac{1}{m}$, we have:
  
 $$ e^{a\omega(\frac{1}{\varepsilon})}
  |< g_\varepsilon, \rho>| \,\leq\,
 C(\int e^{2l\omega(\xi)} |\hat{\rho}(\xi)|^2 d\xi)^{\frac{1}{2}} .$$
  By taking for any $\varepsilon\leq\frac{1}{m} $ for $\rho$,

 $ \rho = g_\varepsilon\, $ we obtain (using also the regularity of $g$\,)that there exists $\,h\,>\,0$ such that:
 \begin{equation*}
 <g_\varepsilon, g_\varepsilon > \,=\,
 < \hat{g_\varepsilon},\hat{g_\varepsilon}>\,\leq\,
 Ce^{-a\omega(\frac{1}{\varepsilon})}
(\int e^{2l\omega(\xi)}|\hat{g_\varepsilon}^2 d\xi )^
{\frac{1}{2}} \,=\,
O(e^{(-a+h)\omega(\frac{1}{\varepsilon })})
 \end{equation*}
  As we can choose $a$ as large as we want, we obtain :
  $$\forall b>0\,,|< g_\varepsilon,g_\varepsilon|> \,=\,
  o(e^{-b\omega(\frac{1}{\varepsilon })})$$
  and this has been proved to be a criterion for 
  $$,\ g\,=\,0$$
  \end{proof}

  In Roumieu case we have a similar result but with a slightly different proof:
   \begin{theorem}
   If $$g = [(g_\varepsilon)] \,\in\,
   \mathcal{G}^{\{\omega\},\infty}$$
    is such that 
 $$ \forall \, \rho \in \,
   \mathcal{D}^{\{\omega)\}}(\Omega) \, 
   (<g_\varepsilon,\rho>) \,\in \,
    \mathcal{N}^{\{\omega\}}(\Omega)\,$$  
    then $$g\,=\,0$$
    \end {theorem}
    \begin{proof}

    As before we can without loss of generality suppose that there exists a compact subset $K$ of $\Omega$ 
  whose interior contains all the supports of all $g_\varepsilon$ .Here again we use the families of $L^2$ seminorms.
   The regularity hypothesis amounts to 
  $$\exists h>0\,\mbox{s.t.}\,\forall k>0,\exists m \in\NN \exists C_k>0
  \mbox{s.t.}\,\forall \,\varepsilon \leq \frac{1}{m}
  (\int e^{2h\omega(\xi)}|\hat{g_\varepsilon}(\xi)|^2 d\xi)^\frac{1}{2} \,\leq\,
  C_ke^{k\omega(\frac{1}{\varepsilon})}$$
  As $\,\mathcal{D}^{\{\omega)\}}(\Omega)\,$ contains the subspace $\mathcal{D}_{K,h}^{(\omega)}\,$ which is a Banach space.  the The hypothesis of action on test fuctions implies that 
  
  \begin{equation*}
 \,\forall \rho \in \mathcal{D}_{K,h}^{(\omega)}\,\exists m\in\NN ,\exists a>0 \,,\mbox{s.t.}\, \forall 
 \,\varepsilon \,\leq \frac{1}{m}
  <g_\varepsilon,\rho> \,=\,
  o(e^{-a\omega(\frac{1}{\varepsilon})})\,.
  \end{equation*}
  
   Define now the countable family of closed symetric subsets 
   of this Banach space defined by thre integers $P,l, m$ 
   
   $$  F_{P,l,m}\,=\{,\rho\,\in \,\mathcal{D}_{K,h}^{(\omega)}\,
   \mbox{s.t.}\,\forall\,\varepsilon\,\leq \,\frac{1}{m}   
  |<g_\varepsilon, \rho >|\,\leq\,
   P e^{-\frac{1}{l}\omega(\frac{1}{\varepsilon})}||\,\}
 $$   
  As the Union of all those closed subsets is the Banach space  $ \, \mathcal{D}_{K,h}^{(\omega)}\,$ we can conclude as previously that there is a ball $B(o,r) $ in one of them and thus conclude that there exist $\,C>0,l>0, m>O\,$ such thatfor any  
  $\,\varepsilon \,\leq\,\frac{1}{m}\,$ the following holds
  
  $$ |<g_\varepsilon,\rho>|\,\leq\,
   C(e^{-\frac{1}{l}\omega(\frac{1}{\varepsilon})})||\rho||_{K,h}\,.
   $$
    Applying this for any $\,\varepsilon \leq\frac{1}{m}\, $  on
     $\,\rho\,=\,g_\varepsilon \,$ we obtain , taking into consideration the Roumieu definition of regularity of the net 
     $\,(g_\varepsilon)\,$ in the inductive presentation ,that for any $\,\sigma\,>\,0\,$  we have :
     $$ \,|<g_\varepsilon,g_\varepsilon>|\, =\,
  o( e^{-\frac{1}{l}   +\sigma)\omega(\frac{1}{\varepsilon})}).$$
  By choosing 
  $\,\sigma\,$ strictly smaller than $\frac{1}{l}$ we obtain the Roumieu condition of the negligibility of the net
  $ (||(g_\varepsilon||_{L_2})\,$ , and this implies the (Roumieu) negligibility of the net $\,(g_\varepsilon)\,$

 \end{proof}

REFERENCES %%%%%%%%%%%%%%%%%%%%

%%%%%%%%%%%%%%%%%%%%%%%%%%%%%%%%%%%%%%%%%%%%

%%%%%%%%%%%%%%%%%%%%%%%%%%%%%%%%%%%%%%%%%%%%%%%%%

\end{document}